\documentclass[12pt]{article}

\usepackage{secdot}
\usepackage{geometry}
\usepackage{amsmath}
\usepackage{theorem}
\usepackage{graphicx}

\newtheorem{proposition}{Proposition}
\newtheorem{lemma}{Lemma}
\newtheorem{corollary}{Corollary}
\newtheorem{theorem}{Theorem}

\newenvironment{hangref}{\begin{list}{}
  {\setlength{\itemsep}{4pt}
  \setlength{\parsep}{0pt} \setlength{\leftmargin}{+\parindent}
  \setlength{\itemindent}{-\parindent}}
  }{\end{list}}

\begin{document}

\begin{center}
{\LARGE A Novel Mathematical Model for \\ the Unique Shortest Path Routing Problem \footnote{Some results in this paper have been presented on the 4th
International Conference on Networking (ICN 2005), Reunion Island, April 2005, INFORMS Annual Meeting 2006, Pittsburgh, PA, USA, November 2006, and the 6th
International Congress on Industrial and Applied Mathematics (ICIAM 2007), Zurich, Switzerland, July, 2007, respectively.}} \\ [12pt]

\mbox{\large Changyong Zhang}  \\

{\footnotesize Department of Mathematics, University of Southern California \\ 3620 South Vermont Avenue, KAP 108, Los Angeles, CA 90089-2532, USA,
changyoz@usc.edu} \\ [6pt]

\end{center}

\baselineskip 20pt plus .3pt minus .1pt

\noindent Link weights are the principal parameters of shortest path routing protocols, the most commonly used protocols for IP networks. The problem of optimally
setting link weights for unique shortest path routing is addressed. Due to the complexity of the constraints involved, there exist challenges to formulate the
problem properly, so that a solution algorithm may be developed which could prove to be more efficient than those already in existence. In this paper, a novel
complete formulation with a polynomial number of constraints is first introduced and then mathematically proved to be correct. It is further illustrated that the
formulation has advantages over a prior one in terms of both constraint structure and model size for a proposed decomposition method to solve the problem.


\section{Introduction}

Shortest path routing protocols such as OSPF (Moy, 1998) are the most widely deployed and commonly used protocols for IP networks. In shortest path routing, each
link is assigned a weight and traffic demands are routed along the shortest paths with respect to link weights, based on a shortest path first algorithm
(Bertsekas and Gallager, 1992). Link weights are hence the principal parameters and an essential problem is to find an appropriate weight set for shortest path
routing.

A simple approach to set link weights is the hop-count method, assigning the weight of each link to one. The length of a path is therefore the number of hops.
Another default approach recommended by Cisco is the inv-cap method, setting the weight of a link inversely proportional to its capacity, without taking traffic
conditions into consideration. More generally, the weight of a link may depend on and be related to its transmission capacity and its traffic load. Accordingly, a
problem worth investigating is the task of finding an optimal weight set for shortest path routing, given a network topology, a projected traffic matrix (Feldmann
et al., 2001), and an objective function.

The problem has two instances, depending on whether multiple shortest paths or only a unique routing path from an origin node to a destination node is allowed.
For the first instance, a number of heuristic methods have been introduced, based on genetic algorithms (Ericsson et al., 2002) and local search methods (Fortz
and Thorup, 2000). For the second instance, the Lagrangian relaxation method (Lin and Wang, 1993) and a local search method (Ramakrishnan and Rodrigues, 2001)
have been proposed. These methods have also resulted in good routing performances by being tested in particular data sets. On the other hand, with these heuristic
methods, the problem is not completely formulated and so in general is not optimally solved. It would be desirable if in average cases optimal solutions could be
obtained for reasonably large data instances from real-world applications.

From a management point of view, unique-path routing uses much simpler routing mechanisms and allows for easier monitoring of traffic flows (Ameur and Gourdin,
2003). Therefore, this paper focuses on the unique-path instance. The problem is referred to as the \textit{unique shortest path routing problem}. It is a
reduction of the integer multi-commodity flow problem (Ahuja et al., 1993).

The problem has been well studied and efforts have been made to formulate the problem mathematically (Ameur and Gourdin, 2003; Zhang and Rodo\v sek, 2005).
Mathematical models have also been developed for other related problems such as network design and routing problems (Bley and Koch, 2002; Holmberg and Yuan,
2004). Most of these models have formulated the corresponding problems completely, whereas they have been either demand-based or path-based and have left space to
further explore the structure properties of the problems, from which more efficient solution methods may be derived.

The main goal of this paper is to mathematically model the problem, which would yield a new exact solution approach for real-world applications in average data
instances. In Section 2, the problem definition is first specified. Two different complete formulations, a new one and an existing one for comparison, are
introduced in Section 3. The new formulation is then mathematically proved to be correct in Section 4. Differences between the two formulations on both constraint
structure and model size are discussed in Section 5. Conclusions are drawn in Section 6.

\section{Problem Definition}

The unique shortest path routing problem is specified as follows: \\
Given
\begin{itemize}
\item A network topology, which is a directed graph structure $\mathcal{G} =(\mathcal{N}, \mathcal{L})$, where
\begin{itemize}
\item $\mathcal{N}$ is a finite set of nodes, each of which represents a router; and
\item $\mathcal{L}$ is a set of directed links, each of which corresponds to a transmission link; (For each $(i, j) \in \mathcal{L}$, $i$ is the starting node,
$j$ is the ending node, and $c_{ij} \ge 0$ is the link capacity.)
\end{itemize}
\item A traffic matrix, which is a set of demands $\mathcal{D}$; (It is assumed that there is at most one demand between each origin-destination pair. For each
demand $k \in \mathcal{D}$, $s_k \in \mathcal{N}$ is the origin node, $t_k \in \mathcal{N}$ is the destination node, and $d_k > 0$ is the required bandwidth.
Accordingly, $\mathcal{S}$ is the set of all origin nodes, $\mathcal{T}_s$ is the set of all destination nodes of demands originating from node $s \in
\mathcal{S}$, and $\mathcal{D}_s$ is the set of all demands originating from node $s \in \mathcal{S}$.)
\item Lower and upper bounds of link weights, which are positive real numbers $w_{\mathrm{min}}$ and $w_{\mathrm{max}}$, respectively; and
\item An objective function, specifically, to maximize the sum of residual capacities,
\end{itemize}
Find an optimal weight set $w_{ij}, (i, j)\in \mathcal{L}$, subject to
\begin{itemize}
\item Flow conservation constraints: for each demand, at each node, the sum of all incoming flows (including the demand bandwidth at the origin node) is equal to
the sum of all outgoing flows (including the demand bandwidth at the destination node);
\item Link capacity constraints: for each link, the load of traffic flows transiting the link does not exceed the capacity of the link;
\item Path uniqueness constraints: each demand has a unique routing path; and
\item Path length constraints: for each demand, the length of each path assigned to route the demand is strictly less than that of any other possible and
unassigned path to route the demand.
\end{itemize}

\section{Problem Formulation}

In this section, the problem is mathematically formulated from two different perspectives, based on the study of the problem properties. For comparison, an
existing model (Zhang and Rodo\v sek, 2005) is introduced in more detail first.

\subsection{A Demand-Based Model} \label{sec:dbm}

According to the characteristics of unique shortest path routing, the routing path of a demand is the shortest among all possible paths. For each link, the
routing path of a demand either traverses the link or not.

Figure \ref{fig:d:PathLength} illustrates the relationships between the lengths of the shortest paths and link weights. Paths in thick lines are routing paths.
The path length and path uniqueness constraints require that the length of the unique shortest path to route a demand is less than that of any other possible path
from the origin node to the destination node.

\begin{figure}
\begin{center}
\includegraphics [width = 0.88 \textwidth] {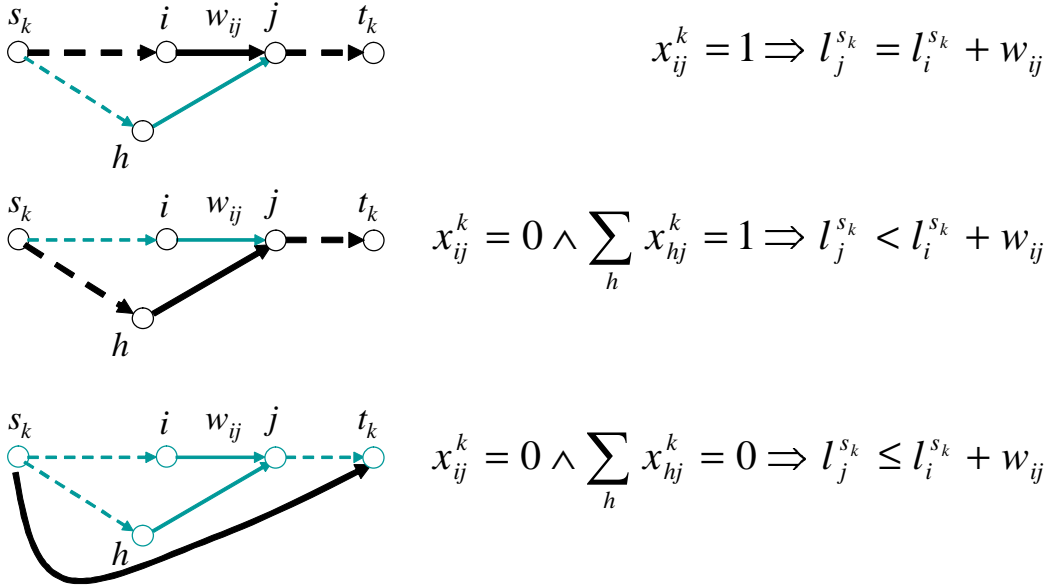}
\caption{Illustration of the path length constraints}
\label{fig:d:PathLength}
\end{center}
\end{figure}

With respect to the constraints, there are three scenarios to be considered:

\begin{itemize}
\item If the routing path of demand $k$ traverses link $(i, j)$, the length of the shortest path from node $s_k$ to node $j$ is the length of the shortest path
from node $s_k$ to node $i$ plus the weight of link $(i, j)$;
\item If the routing path of demand $k$ does not traverse link $(i, j)$ but visits node $j$, the length of the shortest path from node $s_k$ to node $j$ is
strictly less than the sum of the length of the shortest path from node $s_k$ to node $i$ and the weight of link $(i, j)$; (Otherwise, there would be at least two
shortest paths to route demand $k$.)
\item If the routing path of demand $k$ neither traverses link $(i, j)$ nor visits node $j$, the length of the shortest path from node $s_k$ to node $j$ is less
than or equal to the sum of the length of the shortest path from node $s_k$ to node $i$ and the weight of link $(i, j)$.
\end{itemize}

Based on the above observations on the relationships between the length of a shortest path and the weights of links that it traverses, the problem can be
mathematically formulated as a \textit{demand-based} model (DBM) as follows (by defining one routing decision variable for each link-demand pair):

\textit{Routing decision variables}:
\begin{eqnarray} \label{d:rdvar}
x^k_{ij} \in \{0,1\}, \forall k \in \mathcal{D}, \forall (i,j) \in \mathcal{L}
\end{eqnarray}
is equal to $1$ if and only if the routing path of demand $k$ traverses link $(i, j)$. The number of this set of variables is $|\mathcal{D}||\mathcal{L}|$.

\textit{Link weight variables}:
\begin{eqnarray} \label{d:lwvar}
w_{ij} \in [w_{\mathrm{min}}, w_{\mathrm{max}}], \forall (i, j) \in \mathcal{L}
\end{eqnarray}
represents the routing cost of link $(i, j)$. The number of this set of variables is $|\mathcal{L}|$.

\textit{Path length variables}:
\begin{eqnarray} \label{d:plvar}
l^s_i \in [0, +\infty), \forall s \in \mathcal{S}, \forall i \in \mathcal{N}
\end{eqnarray}
denotes the length of the shortest path from origin node $s$ to node $i$. Obviously, $l^{s_k}_{t_k}$ is the length of the shortest path to route demand $k \in
\mathcal{D}$ and $l^s_s = 0$, $\forall s \in \mathcal{S}$. The number of this set of variables is $|\mathcal{S}||\mathcal{N}|$.

\textit{Flow conservation constraints}:
\begin{eqnarray} \label{d:fccon}
\sum_{h: (h, i) \in \mathcal{L}}x^k_{hi} - \sum_{j: (i, j) \in \mathcal{L}}x^k_{ij} = \left \{\begin{array}{rl} -1, & \textrm{if} \ i = s_k \\ 1, & \textrm{if} \
i = t_k \\ 0, & \textrm{otherwise} \end{array} \right., \forall k \in \mathcal{D}, \forall i \in \mathcal{N}
\end{eqnarray}
The number of this set of constraints is $|\mathcal{D}||\mathcal{N}|$.

\textit{Link capacity constraints}:
\begin{eqnarray} \label{d:lccon}
\sum_{k \in \mathcal{D}}d_k x^k_{ij} \leq c_{ij}, \forall (i, j) \in \mathcal{L}
\end{eqnarray}
The number of this set of constraints is $|\mathcal{L}|$.

\textit{Path uniqueness constraints}: under the combined restrictions of the flow conservation constraints and the path length constraints, the constraints are
satisfied automatically.

\textit{Path length constraints}:
\begin{eqnarray} \label{d:plcon}
\left. \begin{array}{l} \displaystyle l^{s_k}_j \le l^{s_k}_i + w_{ij} - \varepsilon \left(\sum_{h: (h, j) \in \mathcal{L}}x^k_{hj} - x^k_{ij} \right) \\
\displaystyle l^{s_k}_j \ge l^{s_k}_i + w_{ij} - M\left(1 - x^k_{ij}\right) \end{array} \right\}, \forall k \in \mathcal{D}, \forall (i, j) \in \mathcal{L}
\end{eqnarray}
where $\varepsilon$ and $M$ are appropriate constants with $0 < \varepsilon \ll M$. The number of this set of constraints is $2|\mathcal{D}||\mathcal{L}|$. By
enumerating all possible values of the routing decision variables $x^k_{ij}, k \in \mathcal{D}, (i, j) \in \mathcal{L}$, it can be verified that the linearized
constraints (\ref{d:plcon}) are identical to those constraints (\ref{d:plcon0}) originally presented in logic forms, as illustrated in Figure
\ref{fig:d:PathLength}.
\begin{eqnarray} \label{d:plcon0}
\left. \begin{array}{r} \displaystyle x^k_{ij} = 0 \wedge \sum_{h: (h, j) \in \mathcal{L}}x^k_{hj} = 0 \Rightarrow l^{s_k}_j \le l^{s_k}_i + w_{ij} \\
\displaystyle x^k_{ij} = 0 \wedge \sum_{h: (h, j) \in \mathcal{L}}x^k_{hj} = 1 \Rightarrow l^{s_k}_j < l^{s_k}_i + w_{ij} \\ \displaystyle x^k_{ij} = 1
\Rightarrow l^{s_k}_j = l^{s_k}_i + w_{ij} \end{array} \right \}, \forall k \in \mathcal{D}, \forall (i, j) \in \mathcal{L}
\end{eqnarray}

\textit{Objective function}:
\begin{eqnarray*}
\max \sum_{(i, j) \in \mathcal{L}}\left(c_{ij} - \sum_{k \in \mathcal{D}}d_k x^k_{ij}\right)
\end{eqnarray*}
which is equivalent to
\begin{eqnarray} \label{d:obj}
\mbox{min} \sum_{(i, j) \in \mathcal{L}}\sum_{k \in \mathcal{D}} d_k x^k_{ij}
\end{eqnarray}

Accordingly, the resulting complete model is presented as

\begin{center}
\textbf{DBM:} $\begin{array}{l} Optimize \ \ (\ref{d:obj}) \\ Subject \ to \ (\ref{d:fccon}), (\ref{d:lccon}), (\ref{d:plcon}), (\ref{d:rdvar}), (\ref{d:lwvar}),
(\ref{d:plvar}) \end{array}$
\end{center}

In the following, an equivalent model to DBM, which will be used to verify the correctness of DBM in Section \ref{sec:vrf}, is derived.

A necessary condition of the unique shortest path routing problem is the \textit{sub-path optimality} requirement, which states that any sub-path of a routing
path is still a unique shortest path (Bley and Koch, 2002). Specifically, given an origin node $s \in \mathcal{S}$ and a node $i \in \mathcal{N}$ where $i \ne s$,
it requires that all demands originating from $s$ and visiting $i$ use the same incoming link to $i$. Mathematically, it is formulated as
\begin{eqnarray} \label{d:spcon}
\sum_{h:(h, i) \in \mathcal{L}} \max_{k \in \mathcal{D}_s} x_{hi}^k \le 1, i \ne s, \forall s \in \mathcal{S}, \forall i \in \mathcal{N}.
\end{eqnarray}

The above constraints can be alternatively presented with linear constraints, by introducing a new variable $y_{ij}^s, s \in \mathcal{S}, (i, j) \in \mathcal{L}$:
\begin{eqnarray*}
y_{ij}^s \ge x_{ij}^k, \forall k \in \mathcal{D}_s, \forall (i, j) \in \mathcal{L} \quad \mathrm{and} \quad \sum_{h:(h, i) \in \mathcal{L}} y_{hi}^s \le 1, i \ne
s, \forall s \in \mathcal{S}, \forall i \in \mathcal{N}.
\end{eqnarray*}

In DBM, the sub-path optimality constraints are not explicitly included. In the following, it is proved that the constraints are implied by the path length
constraints.

\begin{proposition} \label{prop:d:spcon}
The path length constraints in DBM imply the sub-path optimality constraints.
\end{proposition}

\begin{proof}
Suppose there are two demands $k_1, k_2 \in \mathcal{D}$, and $s_{k_1} = s_{k_2} =s$. Assume they would use two disjoint paths to traverse from node $u$ to node
$v$. Demand $k_1$ would use path
\begin{eqnarray*}
P_j = (j_1, j_2) \to (j_2, j_3) \to \ldots \to (j_{n-1}, j_n), \ (j_l, j_{l + 1}) \in \mathcal{L}, l = 1, \ldots, n-1,
\end{eqnarray*}
where $j_1 = u$ and $j_n = v$, and demand $k_2$ would use path
\begin{eqnarray*}
P_i = (i_1, i_2) \to (i_2, i_3) \to \ldots \to (i_{m-1}, i_m), \ (i_{q}, i_{q + 1}) \in \mathcal{L}, q = 1, \ldots, m-1,
\end{eqnarray*}
where $i_1 = u$ and $i_m = v$.

Then, by the definition of the routing decision variables,
\begin{eqnarray*}
x^{k_1}_{j_l j_{l+1}} = 1 \ \textrm{and} \ x^{k_2}_{j_l j_{l+1}} = 0, \forall (j_l, j_{l+1}) \in P_j, l = 1, \ldots, n-1.
\end{eqnarray*}

As a result, according to constraints (\ref{d:plcon0}), on one hand, since $\forall k \in \mathcal{D}, \forall (i, j) \in \mathcal{L}$, $x^k_{ij} = 1 \Rightarrow
l^{s_k}_j = l^{s_k}_i + w_{ij}$, considering demand $k_1$,
\begin{eqnarray} \label{d:spcon:1}
l^{s}_{v} & = & l^{s}_{u} + l_{j_n}^u \nonumber \\ & = & l^{s}_{u} + l_{j_{n-1}}^u + w_{j_{n-1} j_n} \nonumber \\ & = & l^{s}_{u} + l_{j_{n-2}}^u + w_{j_{n-2}
j_{n-1}} + w_{j_{n-1} j_n} \nonumber \\ & = & l^{s}_{u} + l_{j_1}^u + w_{j_1 j_2} + \dots + w_{j_{n-1} j_n} \nonumber \\ & = & l^{s}_{u} + l_{P_j}.
\end{eqnarray}
On the other hand, since $\forall k \in \mathcal{D}, \forall (i, j) \in \mathcal{L}$, $x^k_{ij} = 0 \wedge \sum_{h: (h, j) \in \mathcal{L}}x^k_{hj} = 1
\Rightarrow l^{s_k}_j < l^{s_k}_i + w_{ij}$ and $x^k_{ij} = 0 \Rightarrow l^{s_k}_j \le l^{s_k}_i + w_{ij}$, considering demand $k_2$,
\begin{eqnarray} \label{d:spcon:2}
l^{s}_{v} & = & l^{s}_{u} + l_{j_n}^u \nonumber \\ & < & l^{s}_{u} + l_{j_{n-1}}^u + w_{j_{n-1} j_n} \nonumber \\ & \le & l^{s}_{u} + l_{j_{n-2}}^u + w_{j_{n-2}
j_{n-1}} + w_{j_{n-1} j_n} \nonumber \\ & \le & l^{s}_{u} + l_{j_1}^u + w_{j_1 j_2} + \dots + w_{j_{n-1} j_n} \nonumber \\ & = & l^{s}_{u} + l_{P_j}.
\end{eqnarray}

Contradiction then follows between (\ref{d:spcon:1}) and (\ref{d:spcon:2}), which means the two demands cannot be routed over two different paths between two
shared nodes. It is hence proved that the sub-path optimality constraints are satisfied.
\end{proof}

By Proposition~\ref{prop:d:spcon}, DBM is equivalent to the following model DBM$^\prime$.

\begin{center}
\textbf{DBM$^\prime$:} $\begin{array}{l} Optimize \ \ (\ref{d:obj}) \\ Subject \ to \ (\ref{d:fccon}), (\ref{d:lccon}), (\ref{d:plcon}), (\ref{d:spcon}),
(\ref{d:rdvar}), (\ref{d:lwvar}), (\ref{d:plvar}) \end{array}$
\end{center}

\subsection{An Origin-Based Model}

In Section \ref{sec:dbm}, the unique shortest path routing problem is formulated as a demand-based model, which defines one routing decision variable for each
link-demand pair. Based on the study of solution properties of the problem, it can be found that all routing paths of demands originating from the same node
constitute a tree, rooted at the origin node. Accordingly, a more natural formulation for the problem is to define one routing decision variable for each
link-origin pair. For example, in Figure \ref{fig:o:tree}, instead of defining three routing decision variables for link $(i, j)$, one for each of the three
demands sharing the same origin node $s$, the new formulation defines only one routing decision variable for link $(i, j)$, associated with origin node $s$.
\begin{figure}[h!]
\begin{center}
\includegraphics [width = 0.66\textwidth] {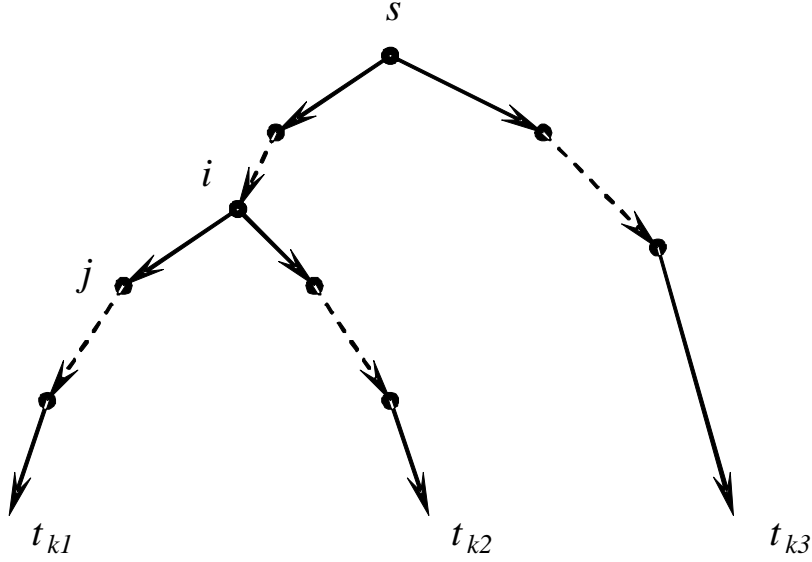}
\caption{Illustration of the origin-based model} \label{fig:o:tree}
\end{center}
\end{figure}

Based on the above observations, an \textit{origin-based} model (OBM) for the problem is formulated as follows.

\textit{Routing decision variables}:
\begin{eqnarray} \label{o:rdvar}
y^s_{ij} \in \{0,1\}, \forall s \in \mathcal{S}, \forall (i, j) \in \mathcal{L}
\end{eqnarray}
is equal to $1$ if and only if the routing paths of at least one of the demands originating from node $s$ traverse link $(i, j)$. The number of this set of
variables is $|\mathcal{S}||\mathcal{L}|$.

\textit{Auxiliary flow variables}:
\begin{eqnarray} \label{o:afvar}
f^s_{ij} \in [0, +\infty), \forall s \in \mathcal{S}, \forall (i, j) \in \mathcal{L}
\end{eqnarray}
represents the load of traffic flows originating from node $s$ and traversing link $(i, j)$. The number of this set of variables is $|\mathcal{S}||\mathcal{L}|$.

\textit{Link weight variables}:
\begin{eqnarray} \label{o:lwvar}
w_{ij} \in [w_{\mathrm{min}}, w_{\mathrm{max}}], \forall (i, j) \in \mathcal{L}
\end{eqnarray}
represents the routing cost of link $(i,j)$. The number of this set of variables is $|\mathcal{L}|$.

\textit{Path length variables}:
\begin{eqnarray} \label{o:plvar}
l^s_i \in [0, +\infty), \forall s \in \mathcal{S}, \forall i \in \mathcal{N}
\end{eqnarray}
denotes the length of the shortest path from origin node $s$ to node $i$. In particular, $l^s_s = 0, \forall s \in \mathcal{S}$. The number of this set of
variables is $|\mathcal{S}||\mathcal{N}|$.

\textit{Flow conservation constraints}: for each tree, at the root node, the difference between the sum of outgoing flows and the sum of incoming flows is the sum
of bandwidths of all demands originating from the root node; at the destination node of each demand originating from the root node, the difference between the sum
of incoming flows and the sum of outgoing flows is the bandwidth of the demand; and the sum of incoming flows is equal to that of outgoing flows at any other
node.
\begin{eqnarray} \label{o:fccon}
\sum_{h: (h, i) \in \mathcal{L}}f^s_{hi} - \sum_{j: (i,j) \in \mathcal{L}}f^s_{ij} = \left \{\begin{array}{rl} -d_s, & \textrm{if} \ i = s \\ d_k, & \textrm{if} \
i = t_k, \forall k \in \mathcal{D}_s \\ 0, & \textrm{otherwise} \end{array} \right., \forall s \in \mathcal{S}, \forall i \in \mathcal{N}
\end{eqnarray}
where $d_s = \sum_{k \in \mathcal{D}_s} d_k$. The number of this set of constraints is $|\mathcal{S}||\mathcal{N}|$.

\textit{Flow bound constraints}: for each tree, the total flow load over each link does not exceed the sum of all demand bandwidths originating from the root node
and it is equal to zero if no demand originating from the root node is routed over the link.
\begin{eqnarray} \label{o:fbcon}
f^s_{ij} \le y^s_{ij} \sum_{k \in \mathcal{D}_s}d_k, \forall s \in \mathcal{S}, \forall (i, j) \in \mathcal{L}
\end{eqnarray}
The number of this set of constraints is $|\mathcal{S}||\mathcal{L}|$.

\textit{Link capacity constraints}:
\begin{eqnarray} \label{o:lccon}
\sum_{s \in \mathcal{S}} f^s_{ij} \leq c_{ij}, \forall (i, j) \in \mathcal{L}
\end{eqnarray}
The number of this set of constraints is $|\mathcal{L}|$.

\textit{Path uniqueness constraints}: for each tree, the number of incoming links with non-zero flows is equal to zero at the origin node; the number of incoming
links with non-zero flows is equal to one at the destination node of each demand originating from the root node; and the number of incoming links with non-zero
flows does not exceed one at any other node.
\begin{eqnarray} \label{o:pucon}
\sum_{h: (h, i) \in \mathcal{L}} y^s_{hi} \left\{\begin{array}{ll} = 0, & \textrm{if} \ i = s \\ = 1, & \textrm{if} \ i \in \mathcal{T}_s \\ \le 1, &
\textrm{otherwise} \end{array} \right., \forall s \in \mathcal{S}, \forall i \in \mathcal{N}
\end{eqnarray}
The number of this set of constraints is $|\mathcal{S}||\mathcal{N}|$.

\textit{Path length constraints}: for each tree, the length of the unique shortest path to route a demand originating from the root node is less than that of any
other possible path from the origin node to the destination node.
\begin{eqnarray} \label{o:plcon0}
\left. \begin{array}{r} \displaystyle y^s_{ij} = 0 \wedge \sum_{h: (h, j) \in \mathcal{L}}y^s_{hj} = 0 \Rightarrow l^s_j \le l^s_i + w_{ij} \\ \displaystyle
y^s_{ij} = 0 \wedge \sum_{h: (h, j) \in \mathcal{L}}y^s_{hj} = 1 \Rightarrow l^s_j < l^s_i + w_{ij} \\ \displaystyle y^s_{ij} = 1 \Rightarrow l^s_j = l^s_i +
w_{ij} \end{array} \right\}, \forall s \in \mathcal{S}, \forall (i, j) \in \mathcal{L}
\end{eqnarray}
The logic constraints (\ref{o:plcon0}) can be linearized as follows:
\begin{eqnarray} \label{o:plcon}
\left. \begin{array}{l} \displaystyle l^s_j \le l^s_i + w_{ij} - \varepsilon \left(\sum_{h: (h, j) \in \mathcal{L}}y^s_{hj} - y^s_{ij}\right) \\ \displaystyle
l^s_j \ge l^s_i + w_{ij} - M(1 - y^s_{ij}) \end{array} \right\}, \forall s \in \mathcal{S}, \forall (i, j) \in \mathcal{L}
\end{eqnarray}
where $\varepsilon$ and $M$ are appropriate constants with $0 < \varepsilon \ll M$. The number of this set of constraints is $2|\mathcal{S}||\mathcal{L}|$.

\textit{Objective function}:
\begin{eqnarray*}
\max \sum_{(i, j) \in \mathcal{L}}\left(c_{ij} - \sum_{s \in \mathcal{S}} f^s_{ij}\right)
\end{eqnarray*}
which is equivalent to
\begin{eqnarray} \label{o:obj}
\mbox{min} \sum_{(i, j) \in \mathcal{L}}\sum_{s \in \mathcal{S}} f^s_{ij}
\end{eqnarray}

Accordingly, the resulting complete model is presented as

\begin{center}
\textbf{OBM:} $\begin{array}{l} Optimize \ \ (\ref{o:obj}) \\ Subject \ to \ (\ref{o:fccon}), (\ref{o:fbcon}), (\ref{o:lccon}), (\ref{o:pucon}), (\ref{o:plcon}),
(\ref{o:rdvar}), (\ref{o:afvar}), (\ref{o:lwvar}), (\ref{o:plvar}) \end{array}$
\end{center}

\section{Model Verification} \label{sec:vrf}

The correctness of OBM is verified in this section. The verification is divided into two steps.

First, DBM is verified to be a correct model of the unique shortest path routing problem.

Second, OBM and DBM are mathematically proved to be equivalent concerning both the feasibility and the optimality of the problem, which implies that OBM is also a
correct model of the problem.

Since a relaxation of the unique shortest path routing problem has been thoroughly studied and a corresponding demand-based formulation for the relaxed problem
has been recognized as being correct, the verification of DBM is built on the correctness of the formulation for the relaxed problem, the \textit{integer
multi-commodity flow problem} (Ahuja et al., 1993) with sub-path optimality condition. Also, the equivalence between OBM and DBM is demonstrated based on the
proof of the equivalence between two corresponding models of the relaxation, which are denoted as RDBM and ROBM, respectively.

\begin{center}
$\begin{array}{l} \textbf{RDBM:} \begin{array}{l} Optimize \ \ (\ref{d:obj}) \\ Subject \ to \ (\ref{d:fccon}), (\ref{d:lccon}), (\ref{d:spcon}), (\ref{d:rdvar})
\end{array}
\vspace{1 ex} \\
\textbf{ROBM:} \begin{array}{l} Optimize \ \ (\ref{o:obj}) \\ Subject \ to \ (\ref{o:fccon}), (\ref{o:fbcon}), (\ref{o:lccon}), (\ref{o:pucon}), (\ref{o:rdvar}),
(\ref{o:afvar}) \end{array} \end{array}$
\end{center}

As can be noted, RDBM is actually a relaxation of DBM$^\prime$, which is equivalent to DBM by Proposition~\ref{prop:d:spcon}.

\subsection{Correctness of DBM}

To verify the correctness of DBM, that of the equivalent model DBM$^\prime$ is proved. Apparently, the difference between DBM$^\prime$ and RDBM lies in the path
length constraints (\ref{d:plcon}) and the additional link weight variables (\ref{d:lwvar}) as well as the path length variables (\ref{d:plvar}). In order to
verify that DBM$^\prime$ formulates the unique shortest path routing problem correctly, constraints (\ref{d:plcon}) in DBM$^\prime$ are proved to represent
correctly the additional path length constraints. Specifically, the following two statements are demonstrated to be correct. In DBM$^\prime$, the path length
constraints (\ref{d:plcon}), combined with the flow conservation constraints (\ref{d:fccon}), guarantee that:
\begin{enumerate}
\item The routing path of each demand is a shortest path; and
\item The routing path of each demand is a unique path.
\end{enumerate}

The two statements are verified in Proposition~\ref{prop:d:sp} and Proposition~\ref{prop:d:up}, respectively. Proposition~\ref{prop:d:sp} is proved by
demonstrating that the length of the routing path of each demand is less than or equal to that of any other possible path. The proof of
Proposition~\ref{prop:d:up} is built on, with two lemmas, the satisfaction of the single-path requirement by the relaxed problem RDBM, followed by the proof that
the uniqueness requirement is satisfied by DBM$^\prime$. As the original logic constraints (\ref{d:plcon0}) are identical to the linearized constraints
(\ref{d:plcon}), the following proof is based on the original constraints.

\begin{proposition} \label{prop:d:sp}
Each routing path resulting from the solution to DBM$^\prime$ is a shortest path.
\end{proposition}

\begin{proof}
Assume for demand $k$,
\begin{eqnarray*}
P_j = (j_1, j_2) \to (j_2, j_3) \to \ldots \to (j_{n-1}, j_n), \ (j_l, j_{l+1}) \in \mathcal{L}, l = 1, \ldots, n-1,
\end{eqnarray*}
where $j_1 = s_k$ and $j_n = t_k$, is the assigned routing path, and
\begin{eqnarray*}
P_i = (i_1, i_2) \to (i_2, i_3) \to \ldots \to (i_{m-1}, i_m), \ (i_{q}, i_{q+1}) \in \mathcal{L}, q = 1, \ldots, m-1,
\end{eqnarray*}
where $i_1 = s_k$ and $i_m = t_k$, is any other possible and non-assigned path from $s_k$ to $t_k$.

Then, by the definition of the routing decision variables, $x^k_{j_l j_{l+1}} = 1$, $\forall (j_l, j_{l+1}) \in P_j$, $l = 1, \ldots, n-1$ and $\exists (i_q,
i_{q+1}) \in P_i$, $x^k_{i_q i_{q+1}} = 0$, $q \in \{1, 2, \ldots ,{m-1}\}$.

As a result, according to constraints (\ref{d:plcon0}), on one hand, since $\forall (i, j) \in \mathcal{L}$, $x^k_{ij} = 1 \Rightarrow l^{s_k}_j = l^{s_k}_i +
w_{ij}$,
\begin{eqnarray} \label{d:pj}
l^{s_k}_{t_k} = l^{s_k}_{j_n} = l^{s_k}_{j_{n-1}} + w_{j_{n-1} j_n} = l^{s_k}_{j_1} + w_{j_1 j_2} + \ldots + w_{j_{n-1} j_n} = l_{P_j}.
\end{eqnarray}

On the other hand, since $\forall (i, j) \in \mathcal{L}$, $x^k_{ij} = 0 \Rightarrow l^{s_k}_j \le l^{s_k}_i + w_{ij}$,
\begin{eqnarray} \label{d:pi}
l^{s_k}_{t_k} = l^{s_k}_{i_m} \le l^{s_k}_{i_{m-1}} + w_{i_{m-1} i_m} \le l^{s_k}_{i_1} + w_{i_1 i_2}+ \ldots + w_{i_{m-1} i_m} = l_{P_i}.
\end{eqnarray}

It thus follows that $l_{P_j} \le l_{P_i}$ from (\ref{d:pj}) and (\ref{d:pi}). Hence, the length of the routing path $P_j$ is less than or equal to the length of
any other possible path $P_i$ to route demand $k$. It is therefore proved that the routing path $P_j$ is the shortest one.
\end{proof}

In order to prove that DBM$^\prime$ satisfies the path uniqueness constraints, the constraints are demonstrated to be satisfied by the relaxed problem RDBM first.

\begin{lemma} \label{lem:d:fl}
No optimal solution to RDBM contains flow loops.
\end{lemma}

\begin{proof}
Suppose ${x^*}_{ij}^{k}, k \in \mathcal{D}, (i, j) \in \mathcal{L}$ is an optimal solution to RDBM, then all constraints (\ref{d:fccon}), (\ref{d:lccon}), and
(\ref{d:spcon}) are satisfied by ${x^*}_{ij}^{k}, k \in \mathcal{D}, (i, j) \in \mathcal{L}$. Specifically,
\begin{eqnarray} \label{d:fccon*}
\sum_{h: (h, i) \in \mathcal{L}}{x^*}^{k}_{hi} - \sum_{j: (i, j) \in \mathcal{L}}{x^*}_{ij}^{k} = \left\{\begin{array}{rl} -1, & \textrm{if} \ i = s_k \\ 1, &
\textrm{if} \ i = t_k \\ 0, & \textrm{otherwise} \end{array} \right., \forall k \in \mathcal{D}, \forall i \in \mathcal{N},
\end{eqnarray}
\begin{eqnarray} \label{d:lccon*}
\sum_{k \in \mathcal{D}}d_k {x^*}_{ij}^{k} \leq c_{ij}, \forall (i, j) \in \mathcal{L},
\end{eqnarray}
and
\begin{eqnarray} \label{d:spcon*}
\sum_{h:(h, i) \in \mathcal{L}} \max_{k \in \mathcal{D}_s} {x^*}_{hi}^k \le 1, i \ne s, \forall s \in \mathcal{S}, \forall i \in \mathcal{N}.
\end{eqnarray}

Assume there would exist $k_l \in \mathcal{D}$ and a loop $\mathcal{C}: (j_1, j_2) \to \dots \to (j_{n-1}, j_n) \to (j_n, j_{n+1}), (j_i, j_{i+1}) \in
\mathcal{L}, i = 1, \dots, n$ and $j_{n+1} = j_1$, such that ${x^*}_{j_i j_{i+1}}^{k_l} = 1, i = 1, \dots, n$. Obviously,
\begin{eqnarray} \label{d:fccon:L}
\sum_{h: (h, i) \in \mathcal{C}}{x^*}^{k_l}_{hi} - \sum_{j: (i, j) \in \mathcal{C}}{x^*}_{ij}^{k_l} = 0, \forall i \in \{j_1, \dots, j_n\}.
\end{eqnarray}

Let
\begin{eqnarray}  \label{d:yvar}
{y^*}_{ij}^{k} = \left \{\begin{array}{rl} {x^*}_{ij}^{k}, & \textrm{if} \ k \ne k_l \ \textrm{or} \ (i, j) \notin \mathcal{C} \\ 0, & \textrm{otherwise}
\end{array} \right., \forall k \in \mathcal{D}, \forall (i, j) \in \mathcal{L}.
\end{eqnarray}

Then, ${y^*}_{ij}^{k} \in \{0, 1\}, \forall k \in \mathcal{D}, \forall (i, j) \in \mathcal{L}$ and ${y^*}_{j_i j_{i+1}}^{k_l} = 0, \forall (j_i, j_{i+1}) \in
\mathcal{C}, i = 1, \dots, n$. In addition, according to (\ref{d:fccon*}), (\ref{d:fccon:L}), and (\ref{d:yvar}),
\begin{eqnarray*}
\sum_{h: (h, i) \in \mathcal{L}}{y^*}^{k}_{hi} - \sum_{j: (i, j) \in \mathcal{L}}{y^*}_{ij}^{k} = \sum_{h: (h, i) \in \mathcal{L}}{x^*}^{k}_{hi} - \sum_{j: (i, j)
\in \mathcal{L}}{x^*}_{ij}^{k} = \left\{\begin{array}{rl} -1, & \textrm{if} \ i = s_k \\ 1, & \textrm{if} \ i = t_k \\ 0, & \textrm{otherwise} \ \end{array}
\right., \forall k \in \mathcal{D}, \forall i \in \mathcal{N},
\end{eqnarray*}
according to (\ref{d:lccon*}) and (\ref{d:yvar}),
\begin{eqnarray*}
\sum_{k \in \mathcal{D}}d_k {y^*}_{ij}^{k} \le \sum_{k \in \mathcal{D}}d_k {x^*}_{ij}^{k} \le c_{ij}, \forall (i, j) \in \mathcal{L},
\end{eqnarray*}
and according to (\ref{d:spcon*}) and (\ref{d:yvar}),
\begin{eqnarray*}
\sum_{h:(h, i) \in \mathcal{L}} \max_{k \in \mathcal{D}_s} {y^*}_{hi}^k \le \sum_{h:(h, i) \in \mathcal{L}} \max_{k \in \mathcal{D}_s} {x^*}_{hi}^k \le 1, i \ne
s, \forall s \in \mathcal{S}, \forall i \in \mathcal{N}.
\end{eqnarray*}

Hence, ${y^*}_{ij}^{k}, k \in \mathcal{D}, (i, j) \in \mathcal{L}$ satisfies all constraints (\ref{d:fccon}), (\ref{d:lccon}), and (\ref{d:spcon}). It is
therefore a feasible solution to RDBM. Furthermore, according to (\ref{d:yvar}), since $\forall k \in \mathcal{D}, d_k > 0$,
\begin{eqnarray*}
\displaystyle \sum_{(i, j) \in \mathcal{L}} \sum_{k \in \mathcal{D}}d_k {y^*}_{ij}^{k} & = & \sum_{(i, j) \in \mathcal{L} \setminus \mathcal{C} \lor k \in
\mathcal{D}: k \ne k_l} d_k {y^*}_{ij}^{k} + \sum_{(i, j) \in \mathcal{C}} d_{k_l} {y^*}_{ij}^{k_l} \\ & = & \sum_{(i, j) \in \mathcal{L} \setminus \mathcal{C}
\lor k \in \mathcal{D}: k \ne k_l} d_k {y^*}_{ij}^{k} \\ & = & \sum_{(i, j) \in \mathcal{L} \setminus \mathcal{C} \lor k \in \mathcal{D}: k \ne k_l} d_k
{x^*}_{ij}^{k} \\ & < & \sum_{(i, j) \in \mathcal{L} \setminus \mathcal{C} \lor k \in \mathcal{D}: k \ne k_l} d_k {x^*}_{ij}^{k} + \sum_{(i, j) \in \mathcal{C}}
d_{k_l} {x^*}_{ij}^{k_l} \\ & = & \sum_{(i, j) \in \mathcal{L}} \sum_{k \in \mathcal{D}}d_k {x^*}_{ij}^{k}.
\end{eqnarray*}

Then, ${x^*}_{ij}^{k}, k \in \mathcal{D}, (i, j) \in \mathcal{L}$ could not be an optimal solution to RDBM, which results in contradiction. It hence proves that
no optimal solution to RDBM contains flow loops.
\end{proof}

Based on Lemma~\ref{lem:d:fl}, the path uniqueness constraints are proved to be satisfied by RDBM in Lemma~\ref{lem:d:pucon}.

\begin{lemma} \label{lem:d:pucon}
The path uniqueness constraints are satisfied by RDBM.
\end{lemma}

\begin{proof}
By Lemma~\ref{lem:d:fl}, in RDBM, the flow conservation constraints (\ref{d:fccon}) at origin nodes are equivalent to:
\begin{eqnarray} \label{d:fccon:s}
\sum_{h: (h, i) \in \mathcal{L}}x^k_{hi} = 0 \ \mathrm{and} \ \sum_{j: (i, j) \in \mathcal{L}}x^k_{ij} = 1, i = s_k, \forall k \in \mathcal{D}.
\end{eqnarray}

Similarly, the flow conservation constraints (\ref{d:fccon}) at destination nodes are equivalent to:
\begin{eqnarray} \label{d:fccon:t}
\sum_{h: (h, i) \in \mathcal{L}}x^k_{hi} = 1 \ \mathrm{and} \ \sum_{j: (i, j) \in \mathcal{L}}x^k_{ij} = 0, i = t_k, \forall k \in \mathcal{D}.
\end{eqnarray}

Since $x_{ij}^{k} \in \{0, 1\}, \forall k \in \mathcal{D}, \forall (i, j) \in \mathcal{L}$, constraints (\ref{d:fccon:s}) restrict that there is one, and only
one, outgoing link with a non-zero flow from the origin node of demand $k$. Similarly, constraints (\ref{d:fccon:t}) restrict that there is one, and only one,
incoming link with a non-zero flow into the destination node of demand $k$. In addition, constraints (\ref{d:fccon}) guarantee that, at each intermediate node,
the number of incoming links with non-zero flows is equal to the number of outgoing links with non-zero flows. Hence, for each demand $k$, the number of routing
paths is no more than one. Therefore, the path uniqueness constraints are satisfied.
\end{proof}

\begin{proposition} \label{prop:d:up}
The path length constraints in DBM$^\prime$ restrict that the resulting shortest path of each demand is a unique path.
\end{proposition}

\begin{proof}
Since DBM$^\prime$ is a reduction of RDBM, the solution to the routing decision variables $x^k_{ij}, k \in \mathcal{D}, (i, j) \in \mathcal{L}$ of DBM$^\prime$ is
also a solution to those of RDBM.

By Lemma~\ref{lem:d:pucon}, there is only one routing path for each demand. Suppose for demand $k$,
\begin{eqnarray*}
P_j = (j_1, j_2) \to (j_2, j_3) \to \ldots \to (j_{n-1}, j_n), \ (j_l, j_{l+1}) \in \mathcal{L}, l = 1, \dots, n-1,
\end{eqnarray*}
where $j_1 = s_k, j_n = t_k$, is the assigned routing path, and
\begin{eqnarray*}
P_i = (i_1, i_2) \to (i_2, i_3) \to \ldots \to (i_{m-1}, i_m), \ (i_q, i_{q+1}) \in \mathcal{L}, q = 1, \dots, m-1,
\end{eqnarray*}
where $i_1 = s_k$ and $i_m = t_k$, is any other possible and non-assigned path from $s_k$ to $t_k$.

Then, by the definition of the routing decision variables, $x^k_{j_l j_{l+1}} = 1, \forall (j_l, j_{l+1}) \in P_j, l = 1, \ldots , n-1$.

As a result, according to constraints (\ref{d:plcon0}), since $\forall (i, j) \in \mathcal{L}$, $x^k_{ij} = 1 \Rightarrow l^{s_k}_j = l^{s_k}_i + w_{ij}$,
\begin{eqnarray} \label{d:pucon1}
l^{s_k}_{t_k} = l^{s_k}_{j_n} = l^{s_k}_{j_{n-1}} + w_{j_{n-1} j_n} = l^{s_k}_{j_1} + w_{j_1 j_2} + \ldots + w_{j_{n-1} j_n} = l_{P_j}.
\end{eqnarray}

As both $P_j$ and $P_i$ are paths between $s_k$ and $t_k$, they finally merge at one node. Assume it is node $r$ and $r = j_p = i_q, p \in \{2, 3, \ldots, n-1,
n\}, q \in \{2, 3, \ldots, m-1, m\}$. Then, by the definition of the routing decision variables, $x^k_{j_{p-1} r} = 1$ and $x^k_{i_{q-1} r} = 0$.

In addition, on one hand, $\sum_{h:(h, r) \in \mathcal{L}}x^k_{hr} \le 1$. On the other hand, $\sum_{h:(h, r) \in \mathcal{L}}x^k_{hr} \ge x^k_{j_{p-1} r} +
x^k_{i_{q-1} r} = 1$. Hence, $\sum_{h:(h, r) \in \mathcal{L}}x^k_{hr} = 1$.

As a result, according to constraints (\ref{d:plcon0}),
\begin{eqnarray} \label{d:pucon2}
l^{s_k}_{t_k} & = & l^{s_k}_{i_m} \nonumber \\ & = & l^{s_k}_{i_{m-1}} + w_{i_{m-1} i_m} \nonumber \\ & = & l^{s_k}_r + w_{r i_{q+1}} + \ldots + w_{i_{m-1} i_m}
\nonumber \\ & < & l^{s_k}_{i_{q-1}} + w_{i_{q-1} r} + w_{r i_{q+1}} + \ldots + w_{i_{m-1} i_m} \nonumber \\ & \leq & l^{s_k}_{i_1} + w_{i_1 i_2}+ \ldots +
w_{i_{m-1} i_m} \nonumber \\ & = & l_{P_i}.
\end{eqnarray}

It follows that $l_{P_j} < l_{P_i}$ from (\ref{d:pucon1}) and (\ref{d:pucon2}). It is therefore proved that path $P_j$ is the unique shortest path to route demand
$k$.
\end{proof}

\begin{corollary} \label{cor:d:cor}
DBM is a correct model of the unique shortest path routing problem.
\end{corollary}

\begin{proof}
By Proposition~\ref{prop:d:sp} and Proposition~\ref{prop:d:up}, DBM$^\prime$ is a correct model of the unique shortest path routing problem. Hence, as an
equivalent model to DBM$^\prime$, DBM is also a correct model of the problem.
\end{proof}

\subsection{Correctness of OBM}

By the proof of Lemma~\ref{lem:d:pucon}, the flow conservation constraints (\ref{d:fccon}) are identical to:
\begin{eqnarray} \label{d:fccon:}
\left. \begin{array}{rl} \displaystyle \sum_{h: (h, i) \in \mathcal{L}}x^k_{hi} = 0, \sum_{j: (i, j) \in \mathcal{L}}x^k_{ij} = 1, & \textrm{if} \ i = s_k \\
\displaystyle \sum_{h: (h, i) \in \mathcal{L}}x^k_{hi} = 1, \sum_{j: (i, j) \in \mathcal{L}}x^k_{ij} = 0, & \textrm{if} \ i = t_k \\ \displaystyle \sum_{h: (h, i)
\in \mathcal{L}}x^k_{hi} = \sum_{j: (i, j) \in \mathcal{L}}x^k_{ij}, & \textrm{otherwise} \end{array} \right\}, \forall k \in \mathcal{D}, \forall i \in
\mathcal{N}.
\end{eqnarray}
Hence, RDBM is equivalent to the following model:

\begin{center}
\textbf{RDBM$^\prime$:} $\begin{array}{l} Optimize \ \ (\ref{d:obj})
\\ Subject \ to \ (\ref{d:fccon:}), (\ref{d:lccon}), (\ref{d:spcon}),
(\ref{d:rdvar})
\end{array}$
\end{center}

\begin{lemma} \label{lem:o:fccon}
In ROBM, the flow conservation constraints at origin nodes are equivalent to:
\begin{eqnarray*} \label{o:fccon:s}
\sum_{h: (h, i) \in \mathcal{L}}f^s_{hi} = 0 \ \mathrm{and} \ \sum_{j: (i, j) \in \mathcal{L}}f^s_{ij} = \sum_{k \in \mathcal{D}_s} d_k, i = s, \forall s \in
\mathcal{S}.
\end{eqnarray*}
\end{lemma}

\begin{proof}
On one hand, according to (\ref{o:afvar}), if $i = s$,
\begin{eqnarray*}
\sum_{h: (h, i) \in \mathcal{L}}f^s_{hi} \ge 0, \forall s \in \mathcal{S}.
\end{eqnarray*}
On the other hand, according to (\ref{o:fbcon}) and (\ref{o:pucon}), if $i = s$,
\begin{eqnarray*}
\sum_{h: (h, i) \in \mathcal{L}}f^s_{hi} \le \sum_{h: (h, i) \in \mathcal{L}} \left(y^s_{hi} \sum_{k \in \mathcal{D}_s}d_k \right) = \sum_{k \in \mathcal{D}_s}d_k
\sum_{h: (h, i) \in \mathcal{L}}y^s_{hi} = 0, \forall s \in \mathcal{S}.
\end{eqnarray*}
Hence,
\begin{eqnarray*}
\sum_{h: (h, i) \in \mathcal{L}}f^s_{hi} = 0, i = s, \forall s \in \mathcal{S}.
\end{eqnarray*}

It can then be derived directly from (\ref{o:fccon}) that
\begin{eqnarray*}
\sum_{j: (i, j) \in \mathcal{L}}f^s_{ij} = \sum_{k \in \mathcal{D}_s} d_k, i = s, \forall s \in \mathcal{S}.
\end{eqnarray*}

The conclusion is therefore verified.
\end{proof}

By Lemma~\ref{lem:o:fccon}, the flow conservation constraints (\ref{o:fccon}) are identical to:
\begin{eqnarray} \label{o:fccon:}
\left. \begin{array}{rl} \displaystyle \sum_{h: (h, i) \in \mathcal{L}}f^s_{hi} = 0, \sum_{j: (i, j) \in \mathcal{L}}f^s_{ij} = d_s, & \textrm{if} \ i = s \\
\displaystyle \sum_{h: (h, i) \in \mathcal{L}}f^s_{hi} - \sum_{j: (i, j) \in \mathcal{L}}f^s_{ij} = d_k, & \textrm{if} \ i = t_k, \forall k \in \mathcal{D}_s \\
\displaystyle \sum_{h: (h, i) \in \mathcal{L}}f^s_{hi} - \sum_{j: (i, j) \in \mathcal{L}}f^s_{ij} = 0, & \textrm{otherwise} \end{array} \right \}, \forall s \in
\mathcal{S}, \forall i \in \mathcal{N},
\end{eqnarray}
where $d_s = \sum_{k \in \mathcal{D}_s} d_k$.

As a result, ROBM is equivalent to the following model:

\begin{center}
\textbf{ROBM$^\prime$:} $\begin{array}{l} Optimize \ \ (\ref{o:obj}) \\ Subject \ to \ (\ref{o:fccon:}), (\ref{o:fbcon}), (\ref{o:lccon}), (\ref{o:pucon}),
(\ref{o:rdvar}), (\ref{o:afvar}) \end{array}$
\end{center}

In the following, ROBM and RDBM are proved to be equivalent concerning the feasibility of the relaxed problem. The proof is heavily based on the verification of
the equivalence between the two respectively identical models, ROBM$^\prime$ and RDBM$^\prime$.

\begin{proposition} \label{prop:feasible:dtoo}
There is a solution to ROBM if RDBM is feasible.
\end{proposition}

\begin{proof}
Suppose ${x^*}_{ij}^{k}, k \in \mathcal{D}, (i, j) \in \mathcal{L}$ is a feasible solution to RDBM, and so to RDBM$^\prime$. Then, all constraints
(\ref{d:fccon:}), (\ref{d:lccon}), and (\ref{d:spcon}) are satisfied by ${x^*}_{ij}^{k}, k \in \mathcal{D}, (i, j) \in \mathcal{L}$.

Let
\begin{eqnarray} \label{d:rdvar:}
{y^*}_{ij}^{s} = \max_{k \in \mathcal{D}_s} {x^*}_{ij}^{k}, \forall s \in \mathcal{S}, \forall (i, j) \in \mathcal{L},
\end{eqnarray}
and
\begin{eqnarray} \label{d:afvar:}
{f^*}_{ij}^{s} = \sum_{k \in \mathcal{D}_s} d_k {x^*}_{ij}^{k}, \forall s \in \mathcal{S}, \forall (i, j) \in \mathcal{L}.
\end{eqnarray}

Obviously, ${y^*}_{ij}^{s} \in \{0, 1\}$ and ${f^*}_{ij}^{s} \in [0, +\infty)$, $\forall s \in \mathcal{S}, \forall (i, j) \in \mathcal{L}$.

According to (\ref{d:fccon:}), $\forall k \in \mathcal{D}$, if $i = s_k$, $\sum_{h: (h, i) \in \mathcal{L}}{x^*}^k_{hi} = 0$ and $\sum_{j: (i, j) \in
\mathcal{L}}{x^*}^k_{ij} = 1$. Then, according to (\ref{d:afvar:}), $\forall s \in \mathcal{S}$, if $i = s$,
\begin{eqnarray*}
\sum_{h: (h, i) \in \mathcal{L}} {f^*}_{hi}^{s} = \sum_{h: (h, i) \in \mathcal{L}} \sum_{k \in \mathcal{D}_s} d_k {x^*}_{hi}^{k} = \sum_{k \in \mathcal{D}_s} d_k
\sum_{h: (h, i) \in \mathcal{L}} {x^*}_{hi}^{k} = 0,
\end{eqnarray*}
and
\begin{eqnarray*}
\sum_{j: (i, j) \in \mathcal{L}} {f^*}_{ij}^{s} = \sum_{j: (i, j) \in \mathcal{L}} \sum_{k \in \mathcal{D}_s} d_k {x^*}_{ij}^{k} = \sum_{k \in \mathcal{D}_s} d_k
\sum_{j: (i, j) \in \mathcal{L}} {x^*}_{ij}^{k} = \sum_{k \in \mathcal{D}_s} d_k.
\end{eqnarray*}
Hence, at original nodes, constraints (\ref{o:fccon:}) in ROBM$^\prime$ are satisfied by ${f^*}_{ij}^{s}, s \in \mathcal{S}, (i, j) \in \mathcal{L}$.

Also, according to (\ref{d:fccon:}), $\forall k_l \in \mathcal{D}$, if $i = t_{k_l}$, $\sum_{h: (h, i) \in \mathcal{L}}{x^*}^{k_l}_{hi} - \sum_{j: (i,j) \in
\mathcal{L}}{x^*}^{k_l}_{ij} = 1$ and $\forall k \in \mathcal{D}_{s_{k_l}}, k \ne k_l$, $\sum_{h: (h, i) \in \mathcal{L}}{x^*}^k_{hi} - \sum_{j: (i,j) \in
\mathcal{L}}{x^*}^k_{ij} = 0$. Then, according to (\ref{d:afvar:}), $\forall k_l \in \mathcal{D}$, if $s = s_{k_l}$ and $i = t_{k_l}$,
\begin{eqnarray*}
\sum_{h: (h, i) \in \mathcal{L}} {f^*}_{hi}^{s} - \sum_{j: (i, j) \in \mathcal{L}} {f^*}_{ij}^{s} & = & \sum_{h: (h, i) \in \mathcal{L}} \sum_{k \in
\mathcal{D}_s} d_k {x^*}_{hi}^{k} - \sum_{j: (i, j) \in \mathcal{L}} \sum_{k \in \mathcal{D}_s} d_k {x^*}_{ij}^{k} \\ & = & \sum_{k \in \mathcal{D}_s} \left(d_k
\sum_{h: (h, i) \in \mathcal{L}} {x^*}_{hi}^{k}\right) - \sum_{k \in \mathcal{D}_s} \left(d_k \sum_{j: (i, j) \in \mathcal{L}} {x^*}_{ij}^{k}\right) \\ & = &
\sum_{k \in \mathcal{D}_s} d_k \left(\sum_{h: (h, i) \in \mathcal{L}} {x^*}_{hi}^{k} - \sum_{j: (i, j) \in \mathcal{L}} {x^*}_{ij}^{k} \right) \\ & = & d_{k_l}
\left(\sum_{h: (h, i) \in \mathcal{L}} {x^*}_{hi}^{k_l} - \sum_{j: (i, j) \in \mathcal{L}} {x^*}_{ij}^{k_l} \right) \\ & & + \sum_{k \in \mathcal{D}_s, k \ne k_l}
d_k \left(\sum_{h: (h, i) \in \mathcal{L}} {x^*}_{hi}^{k} - \sum_{j: (i, j) \in \mathcal{L}} {x^*}_{ij}^{k} \right) \\ & = & d_{k_l}.
\end{eqnarray*}
Hence, at destination nodes, constraints (\ref{o:fccon:}) in ROBM$^\prime$ are satisfied by ${f^*}_{ij}^{s}, s \in \mathcal{S}, (i, j) \in \mathcal{L}$.

Similarly, according to (\ref{d:fccon:}), $\forall k \in \mathcal{D}$, if $i \ne s_k$ and $i \ne t_k$, $\sum_{h: (h, i) \in \mathcal{L}}{x^*}^k_{hi} - \sum_{j:
(i, j) \in \mathcal{L}}{x^*}^k_{ij} = 0$. Then, according to (\ref{d:afvar:}), $\forall s \in \mathcal{D}$, $\forall i \in \mathcal{N}$, if $i \ne s$ and $i
\notin \mathcal{D}_s$,
\begin{eqnarray*}
\sum_{h: (h, i) \in \mathcal{L}} {f^*}_{hi}^{s} - \sum_{j: (i, j) \in \mathcal{L}} {f^*}_{ij}^{s} & = & \sum_{h: (h, i) \in \mathcal{L}} \sum_{k \in
\mathcal{D}_s} d_k {x^*}_{hi}^{k} - \sum_{j: (i, j) \in \mathcal{L}} \sum_{k \in \mathcal{D}_s} d_k {x^*}_{ij}^{k} \\ & = & \sum_{k \in \mathcal{D}_s} \left(d_k
\sum_{h: (h, i) \in \mathcal{L}} {x^*}_{hi}^{k}\right) - \sum_{k \in \mathcal{D}_s} \left(d_k \sum_{j: (i, j) \in \mathcal{L}} {x^*}_{ij}^{k}\right) \\ & = &
\sum_{k \in \mathcal{D}_s} d_k \left(\sum_{h: (h, i) \in \mathcal{L}} {x^*}_{hi}^{k} - \sum_{j: (i, j) \in \mathcal{L}} {x^*}_{ij}^{k} \right) \\ & = & 0.
\end{eqnarray*}
Hence, at other nodes, constraints (\ref{o:fccon:}) in ROBM$^\prime$ are satisfied by ${f^*}_{ij}^{s}, s \in \mathcal{S}, (i, j) \in \mathcal{L}$.

According to (\ref{d:rdvar:}) and (\ref{d:afvar:}), $\forall s \in \mathcal{S}, \forall (i, j) \in \mathcal{L}$,
\begin{eqnarray*}
{f^*}_{ij}^{s} = \sum_{k \in \mathcal{D}_s} d_k {x^*}_{ij}^{k} \le \sum_{k \in \mathcal{D}_s} d_k \max_{k \in \mathcal{D}_s} {x^*}_{ij}^{k} = \max_{k \in
\mathcal{D}_s} {x^*}_{ij}^{k} \sum_{k \in \mathcal{D}_s} d_k = {y^*}_{ij}^{s} \sum_{k \in \mathcal{D}_s} d_k.
\end{eqnarray*}
Constraints (\ref{o:fbcon}) in ROBM$^\prime$ are then satisfied by ${y^*}_{ij}^{s}, {f^*}_{ij}^{s}, s \in \mathcal{S}, (i, j) \in \mathcal{L}$.

According to (\ref{d:afvar:}) and (\ref{d:lccon}), $\forall (i, j) \in \mathcal{L}$,
\begin{eqnarray*}
\sum_{s \in \mathcal{S}} {f^*}_{ij}^{s} = \sum_{s \in \mathcal{S}} \sum_{k \in \mathcal{D}_s} d_k {x^*}_{ij}^{k} = \sum_{k \in \mathcal{D}} d_k {x^*}_{ij}^{k} \le
c_{ij}.
\end{eqnarray*}
Constraints (\ref{o:lccon}) in ROBM$^\prime$ are thus satisfied by ${f^*}_{ij}^{s}, s \in \mathcal{S}, (i, j) \in \mathcal{L}$.

According to (\ref{d:fccon:}), $\forall k \in \mathcal{D}$, if $i = s_k$, $\sum_{h: (h, i) \in \mathcal{L}}{x^*}^k_{hi} = 0$ and so $\forall k \in \mathcal{D}$,
$\forall (h, i) \in \mathcal{L}$, if $i = s_k$, ${x^*}^k_{h i} = 0$. Then, according to (\ref{d:rdvar:}), $\forall s \in \mathcal{S}$, if $i = s$,
\begin{eqnarray*}
\sum_{h: (h, i) \in \mathcal{L}} {y^*}_{hi}^{s} = \sum_{h: (h, i) \in \mathcal{L}} \max_{k \in \mathcal{D}_s} {x^*}_{hi}^{k} = \sum_{h: (h, i) \in \mathcal{L}}
\max_{k \in \mathcal{D}_s} 0 = 0.
\end{eqnarray*}
Hence, at original nodes, constraints (\ref{o:pucon}) in ROBM$^\prime$ are satisfied by ${y^*}_{ij}^{s}, s \in \mathcal{S}, (i, j) \in \mathcal{L}$.

Also, according to (\ref{d:fccon:}), $\forall k_l \in \mathcal{D}$, if $i = t_{k_l}$, $\sum_{h: (h, i) \in \mathcal{L}}{x^*}^{k_l}_{hi} = 1$. Then, according to
(\ref{d:rdvar:}), $\forall k_l \in \mathcal{D}$, if $s = s_{k_l}$ and $i = t_{k_l}$,
\begin{eqnarray*}
\sum_{h: (h, i) \in \mathcal{L}} {y^*}_{hi}^{s} = \sum_{h: (h, i) \in \mathcal{L}} \max_{k \in \mathcal{D}_s} {x^*}_{hi}^{k} \ge \sum_{h: (h, i) \in \mathcal{L}}
{x^*}_{hi}^{k_l} = 1.
\end{eqnarray*}
In addition, according to (\ref{d:spcon}), $\forall k_l \in \mathcal{D}$, if $i = t_{k_l}$, $\sum_{h:(h, i) \in \mathcal{L}} \max_{k \in \mathcal{D}_{s_{k_l}}}
{x^*}_{hi}^k \le 1$. Then, according to (\ref{d:rdvar:}), $\forall k_l \in \mathcal{D}$, if $s = s_{k_l}$ and $i = t_{k_l}$,
\begin{eqnarray*}
\sum_{h: (h, i) \in \mathcal{L}} {y^*}_{hi}^{s} = \sum_{h: (h, i) \in \mathcal{L}} \max_{k \in \mathcal{D}_s} {x^*}_{hi}^{k} \le 1.
\end{eqnarray*}
Thus, $\forall k \in \mathcal{D}$, if $s = s_k$ and $i = t_k$,
\begin{eqnarray*}
\sum_{h: (h, i) \in \mathcal{L}} {y^*}_{hi}^{s} = 1.
\end{eqnarray*}
Hence, at destination nodes, constraints (\ref{o:pucon}) in ROBM$^\prime$ are satisfied by ${y^*}_{ij}^{s}, s \in \mathcal{S}, (i, j) \in \mathcal{L}$.

Similarly, according to (\ref{d:spcon}), $\forall s \in \mathcal{S}, \forall i \in \mathcal{N}$, if $i \ne s$, $\sum_{h:(h, i) \in \mathcal{L}} \max_{k \in
\mathcal{D}_s} {x^*}_{hi}^k \le 1$. Then, according to (\ref{d:rdvar:}), $\forall s \in \mathcal{S}, \forall i \in \mathcal{N}$, if $i \ne s$ and $i \notin
\mathcal{T}_s$,
\begin{eqnarray*}
\sum_{h: (h, i) \in \mathcal{L}} {y^*}_{hi}^{s} = \sum_{h: (h, i) \in \mathcal{L}} \max_{k \in \mathcal{D}_s} {x^*}_{hi}^{k} \le 1.
\end{eqnarray*}
Hence, at other nodes, constraints (\ref{o:pucon}) in ROBM$^\prime$ are satisfied by ${y^*}_{ij}^{s}, s \in \mathcal{S}, (i, j) \in \mathcal{L}$.

Since all constraints (\ref{o:fccon:}), (\ref{o:fbcon}), (\ref{o:lccon}), and (\ref{o:pucon}) in ROBM$^\prime$ are satisfied by ${y^*}_{ij}^{s}$ and
${f^*}_{ij}^{s}, s \in \mathcal{S}, (i, j) \in \mathcal{L}$, it is a corresponding feasible solution to ROBM$^\prime$ and also to ROBM, of the feasible solution
to RDBM, ${x^*}_{ij}^{k}, k \in \mathcal{D}, (i, j) \in \mathcal{L}$.
\end{proof}

\begin{proposition} \label{prop:feasible:otod}
There is a solution to RDBM if ROBM is feasible.
\end{proposition}

\begin{proof}
Suppose ${y^*}_{ij}^{s}, {f^*}_{ij}^{s}, s \in \mathcal{S}, (i, j) \in \mathcal{L}$ is a feasible solution to ROBM, and so to ROBM$^\prime$. Then, all constraints
(\ref{o:fccon:}), (\ref{o:fbcon}), (\ref{o:lccon}), and (\ref{o:pucon}) are satisfied by ${y^*}_{ij}^{s}, {f^*}_{ij}^{s}, s \in \mathcal{S}, (i, j) \in
\mathcal{L}$.

Let
\begin{eqnarray*}
{x^*}_{h t_k}^{k} = {y^*}_{h t_k}^{s_k}, \forall k \in \mathcal{D}, \forall (h, t_k) \in \mathcal{L}.
\end{eqnarray*}
According to constraints (\ref{o:pucon}) at destination nodes, $\sum_{h: (h, t_k) \in \mathcal{L}} {y^*}_{h t_k}^{s_k} = 1, \forall k \in \mathcal{D}$. Hence,
$\forall k \in \mathcal{D}, \exists h: (h, t_k) \in \mathcal{L}$, such that ${y^*}_{h t_k}^{s_k} = 1$.

$\forall k \in \mathcal{D}, \forall (i, j) \in \mathcal{L}$, if $(i, j) \ne (i, t_k)$, ${x^*}_{ij}^{k}$ is assigned as follows:
\begin{eqnarray}\label{assigning}
\begin{array}{lll} \textbf{Initialize} & {x^*}_{ij}^{k} \gets 0, & \forall k \in \mathcal{D}, \forall (i, j) \in \mathcal{L}, (i, j) \ne (i, t_k) \\
\textbf{For} & k \in \mathcal{D} \\ & \textbf{0} & i \gets t_k \\ & \textbf{Do} & \textrm{find} \ h: (h, i) \in \mathcal{L}, \ \textrm{such that} \
{y^*}_{hi}^{s_k} = 1 \\ & & {x^*}_{hi}^{k} \gets 1 \\ & & i \gets h \\ & \textbf{Until} & i = s_k \end{array}
\end{eqnarray}

In the above assigning process (\ref{assigning}), at node $i$ in each iteration of the inner loop, on one hand, according to (\ref{o:fccon:}),
\begin{eqnarray*}
\sum_{h: (h, i) \in \mathcal{L}} {f^*}_{hi}^{s_k} \ge \sum_{j: (t_k, j) \in \mathcal{L}} {f^*}_{t_k j}^{s_k} + d_k \ge d_k > 0, \forall k \in \mathcal{D},
\end{eqnarray*}
and according to (\ref{o:fbcon}),
\begin{eqnarray*}
\sum_{h: (h, i) \in \mathcal{L}} {f^*}_{hi}^{s_k} \le \sum_{h: (h, i) \in \mathcal{L}} \left({y^*}_{hi}^{s_k} \sum_{k^\prime \in \mathcal{D}_{s_k}} d_{k^\prime}
\right) = \sum_{k^\prime \in \mathcal{D}_{s_k}} d_{k^\prime} \sum_{h: (h, i) \in \mathcal{L}} {y^*}_{hi}^{s_k}, \forall k \in \mathcal{D}.
\end{eqnarray*}
Hence,
\begin{eqnarray*}
\sum_{h: (h, i) \in \mathcal{L}} {y^*}_{hi}^{s_k} > 0, \forall k \in \mathcal{D}.
\end{eqnarray*}
On the other hand, according to (\ref{o:pucon}),
\begin{eqnarray*}
\sum_{h: (h, i) \in \mathcal{L}} {y^*}_{hi}^{s_k} \le 1, \forall k \in \mathcal{D}.
\end{eqnarray*}
Then, at node $i$ in each iteration of the inner loop,
\begin{eqnarray*}
\sum_{h: (h, i) \in \mathcal{L}} {y^*}_{hi}^{s_k} = 1, \forall k \in \mathcal{D}.
\end{eqnarray*}
Therefore, at node $i$ in each iteration of the inner loop, $\exists h: (h, i) \in \mathcal{L}$, such that ${x^*}_{hi}^{k} = {y^*}_{hi}^{s_k} = 1$. Moreover,
$\forall k \in \mathcal{D}$, according to (\ref{o:fccon:}), the process terminates at node $s_k$.

According to (\ref{assigning}), obviously,
\begin{eqnarray*}
{x^*}_{ij}^{k} \in \{0, 1\} \ \textrm{and} \ {x^*}_{ij}^{k} \le {y^*}_{ij}^{s_k}, \forall k \in \mathcal{D}, \forall (i, j) \in \mathcal{L}.
\end{eqnarray*}

Then, according to (\ref{o:pucon}) at original nodes, if $i = s_k$,

\begin{eqnarray*}
\sum_{h: (h, i) \in \mathcal{L}} {x^*}_{hi}^{k} \le \sum_{h: (h, i) \in \mathcal{L}} {y^*}_{hi}^{s_k} = 0, \forall k \in \mathcal{D}.
\end{eqnarray*}
In addition, the assigning process terminates at node $s_k$, $\forall k \in \mathcal{D}$, and so if $i = s_k$,
\begin{eqnarray*}
\sum_{j: (i, j) \in \mathcal{L}} {x^*}_{ij}^{k} = 1, \forall k \in \mathcal{D}.
\end{eqnarray*}
Hence, at original nodes, constraints (\ref{d:fccon:}) in RDBM$^\prime$ are satisfied by ${x^*}_{ij}^{k}, k \in \mathcal{D}, (i, j) \in \mathcal{L}$.

According to (\ref{o:pucon}) at destination nodes, if $i = t_k$,
\begin{eqnarray*}
\sum_{h: (h, i) \in \mathcal{L}} {x^*}_{hi}^{k} = \sum_{h: (h, i) \in \mathcal{L}} {y^*}_{hi}^{s_k} = 1, \forall k \in \mathcal{D}.
\end{eqnarray*}
Furthermore, according to (\ref{assigning}), if $i = t_k$, ${x^*}_{ij}^{k} = 0$, $\forall k \in \mathcal{D}, \forall (i, j) \in \mathcal{L}$. Then, if $i = t_k$,
\begin{eqnarray*}
\sum_{j: (i, j) \in \mathcal{L}} {x^*}_{ij}^{k} = 0, \forall k \in \mathcal{D}.
\end{eqnarray*}
Hence, at destination nodes, constraints (\ref{d:fccon:}) in RDBM$^\prime$ are satisfied by ${x^*}_{ij}^{k}, k \in \mathcal{D}, (i, j) \in \mathcal{L}$.

According to (\ref{assigning}), at node $i$ in each iteration of the inner loop,
\begin{eqnarray*}
\sum_{h: (h, i) \in \mathcal{L}} {x^*}_{hi}^{k} = \sum_{j: (i, j) \in \mathcal{L}} {x^*}_{ij}^{k} = 1, \forall k \in \mathcal{D},
\end{eqnarray*}
and at any other node $i^\prime \in \mathcal{N}, i^\prime \ne s_k, i^\prime \ne t_k$,
\begin{eqnarray*}
\sum_{h: (h, i^\prime) \in \mathcal{L}} {x^*}_{hi^\prime}^{k} = \sum_{j: (i^\prime, j) \in \mathcal{L}} {x^*}_{i^\prime j}^{k} = 0, \forall k \in \mathcal{D}.
\end{eqnarray*}
Then, if $i \ne s_k, i \ne t_k$,
\begin{eqnarray*}
\sum_{h: (h, i) \in \mathcal{L}} {x^*}_{hi}^{k} - \sum_{j: (i, j) \in \mathcal{L}} {x^*}_{ij}^{k} = 0, \forall k \in \mathcal{D}.
\end{eqnarray*}
Hence, at other nodes, constraints (\ref{d:fccon:}) in RDBM$^\prime$ are satisfied by ${x^*}_{ij}^{k}, k \in \mathcal{D}, (i, j) \in \mathcal{L}$.

According to (\ref{o:fccon:}) and (\ref{o:lccon}),
\begin{eqnarray*}
\sum_{k \in \mathcal{D}} d_k {x^*}_{ij}^{k} = \sum_{s \in \mathcal{S}} \sum_{k \in \mathcal{D}_s} d_k {x^*}_{ij}^{k} \le \sum_{s \in \mathcal{S}} {f^*}_{ij}^{s}
\le c_{ij}, \forall (i, j) \in \mathcal{L}.
\end{eqnarray*}
Hence, constraints (\ref{d:lccon}) in RDBM$^\prime$ are satisfied by ${x^*}_{ij}^{k}, k \in \mathcal{D}, (i, j) \in \mathcal{L}$.

According to (\ref{assigning}), ${x^*}_{ij}^{k} \le {y^*}_{ij}^{s_k}$, $\forall k \in \mathcal{D}, \forall (i, j) \in \mathcal{L}$. Then, according to
(\ref{o:pucon}), if $i \ne s$,
\begin{eqnarray*}
\sum_{h: (h, i) \in \mathcal{L}} \max_{k \in \mathcal{D}_s} {x^*}_{hi}^{k} \le \sum_{h: (h, i) \in \mathcal{L}} \max_{k \in \mathcal{D}_s} {y^*}_{hi}^{s_k} =
\sum_{h: (h, i) \in \mathcal{L}} {y^*}_{hi}^{s} \le 1, \forall s \in \mathcal{S}, \forall i \in \mathcal{N}.
\end{eqnarray*}
Hence, constraints (\ref{d:spcon}) in RDBM$^\prime$ are satisfied by ${x^*}_{ij}^{k}, k \in \mathcal{D}, (i, j) \in \mathcal{L}$.

Since all constraints (\ref{d:fccon:}), (\ref{d:lccon}), and (\ref{d:spcon}) in RDBM$^\prime$ are satisfied by ${x^*}_{ij}^{k}, k \in \mathcal{D}, (i, j) \in
\mathcal{L}$, it is a corresponding feasible solution to RDBM$^\prime$, and so to RDBM, of the feasible solution to ROBM, ${y^*}_{ij}^{s}, {f^*}_{ij}^{s}, s \in
\mathcal{S}, (i, j) \in \mathcal{L}$.
\end{proof}

\begin{theorem} \label{thm:feasible:equ}
ROBM and RDBM are equivalent concerning the feasibility of the relaxed problem.
\end{theorem}

\begin{proof}
The conclusion is derived directly from Proposition~\ref{prop:feasible:dtoo} and Proposition~\ref{prop:feasible:otod}.
\end{proof}

Based on the proof of the equivalence between ROBM and RDBM, the equivalence between OBM and DBM, concerning the feasibility of the unique shortest path routing
problem, is verified as follows.

\begin{proposition} \label{prop:plcon:dtoo}
There is a corresponding solution satisfying the path length constraints in OBM, for each solution satisfying the path length constraints in DBM.
\end{proposition}

\begin{proof}
Suppose ${x^*}_{ij}^{k}, k \in \mathcal{D}, (i, j) \in \mathcal{L}$, ${w^*}_{ij}, (i, j) \in \mathcal{L}$, and ${l^*}_{i}^{s}, s\in \mathcal{S}, i \in
\mathcal{N}$ is a feasible solution to DBM and also to DBM$^\prime$ by Proposition~\ref{prop:d:spcon}. Since RDBM is a relaxation of DBM$^\prime$,
${x^*}_{ij}^{k}, k \in \mathcal{D}, (i, j) \in \mathcal{L}$ is then a feasible solution to RDBM and satisfies all corresponding constraints.

Let
\begin{eqnarray*}
{y^*}_{ij}^{s} = \max_{k \in \mathcal{D}_s} {x^*}_{ij}^{k}, \forall s \in \mathcal{S}, \forall (i, j) \in \mathcal{L}.
\end{eqnarray*}
Obviously, ${y^*}_{ij}^{s} \in \{0, 1\}, \forall s \in \mathcal{S}, \forall (i, j) \in \mathcal{L}$. Also, $\forall s \in \mathcal{S}, \forall (i, j) \in
\mathcal{L}$, there are three cases:
\begin{itemize}
\item Case 1: ${x^*}_{ij}^{k} = 0$ and $\sum_{h: (h, j) \in \mathcal{L}} {x^*}_{hj}^{k} = 0$, $\forall k \in \mathcal{D}_s$;
\item Case 2: ${x^*}_{ij}^{k} = 0$, $\forall k \in \mathcal{D}_s$ and $\exists k_l \in \mathcal{D}_s, \sum_{h: (h, j) \in \mathcal{L}} {x^*}_{hj}^{k_l} = 1$;
\item Case 3: $\exists k_l \in \mathcal{D}_s, {x^*}_{ij}^{k_l} = 1$.
\end{itemize}

For Case 1, on one hand, $\forall k \in \mathcal{D}_s$, constraints (\ref{d:plcon}) can be simplified as follows:
\begin{eqnarray*}
l^{s_k}_j \le l^{s_k}_i + w_{ij} \quad \mathrm{and} \quad l^{s_k}_j \ge l^{s_k}_i + w_{ij} - M.
\end{eqnarray*}
On the other hand, since ${y^*}_{ij}^{s} = \max_{k \in \mathcal{D}_s} {x^*}_{ij}^{k} = 0$ and $\sum_{h: (h, j) \in \mathcal{L}} {y^*}_{hj}^{s} = 0$, constraints
(\ref{o:plcon}) can be simplified as follows:
\begin{eqnarray*}
l^{s}_j \le l^{s}_i + w_{ij} \quad \mathrm{and} \quad l^{s}_j \ge l^{s}_i + w_{ij} - M.
\end{eqnarray*}
Hence, the simplified constraints of (\ref{o:plcon}) are identical to those of (\ref{d:plcon}) for Case 1.

For Case 2, on one hand, constraints (\ref{d:plcon}) can be simplified as follows:
\begin{eqnarray*}
l^{s_{k_l}}_j \le l^{s_{k_l}}_i + w_{ij} - \varepsilon \quad \mathrm{and} \quad l^{s_{k_l}}_j \ge l^{s_{k_l}}_i + w_{ij} - M,
\end{eqnarray*}
and $\forall k \in \mathcal{D}_s$ such that $\sum_{h: (h, j) \in \mathcal{L}} {x^*}_{hj}^{k} = 0$,
\begin{eqnarray*}
l^{s_k}_j \le l^{s_k}_i + w_{ij} \quad \mathrm{and} \quad l^{s_k}_j \ge l^{s_k}_i + w_{ij} - M.
\end{eqnarray*}
On the other hand, since ${y^*}_{ij}^{s} = \max_{k \in \mathcal{D}_s} {x^*}_{ij}^{k} = 0$ but $\sum_{h: (h, j) \in \mathcal{L}} {y^*}_{hj}^{s} = 1$, constraints
(\ref{o:plcon}) can be simplified as follows:
\begin{eqnarray*}
l^{s}_j \le l^{s}_i + w_{ij} - \varepsilon \quad \mathrm{and} \quad l^{s}_j \ge l^{s}_i + w_{ij} - M.
\end{eqnarray*}
Hence, the simplified constraints of (\ref{o:plcon}) are identical to those of (\ref{d:plcon}) for Case 2.

For Case 3, on one hand, according to the sub-path optimality constraints, $\forall k \in \mathcal{D}_s$, if ${x^*}_{ij}^{k} = 0$, $\sum_{h: (h, j) \in
\mathcal{L}} {x^*}_{hj}^{k} = 0$. Then, constraints (\ref{d:plcon}) can be simplified as follows:
\begin{eqnarray*}
l^{s_{k_l}}_j \le l^{s_{k_l}}_i + w_{ij} \quad \mathrm{and} \quad l^{s_{k_l}}_j \ge l^{s_{k_l}}_i + w_{ij},
\end{eqnarray*}
and $\forall k \in \mathcal{D}_s$ such that ${x^*}_{ij}^{k} = 0$,
\begin{eqnarray*}
l^{s_k}_j \le l^{s_k}_i + w_{ij} \quad \mathrm{and} \quad l^{s_k}_j \ge l^{s_k}_i + w_{ij} - M.
\end{eqnarray*}
On the other hand, since ${y^*}_{ij}^{s} = \max_{k \in \mathcal{D}_s} {x^*}_{ij}^{k} = 1$ and $\sum_{h: (h, j) \in \mathcal{L}} {y^*}_{hj}^{s} = 1$, constraints
(\ref{o:plcon}) can be simplified as follows:
\begin{eqnarray*}
l^{s}_j \le l^{s}_i + w_{ij} \quad \mathrm{and} \quad l^{s}_j \ge l^{s}_i + w_{ij}.
\end{eqnarray*}
Hence, the simplified constraints of (\ref{o:plcon}) are identical to those of (\ref{d:plcon}) for Case 3.

Since for all the three cases, ${y^*}_{ij}^{s}, s \in \mathcal{S}, (i, j) \in \mathcal{L}$ results in the same path length constraints for OBM as those resulting
from ${x^*}_{ij}^{k}, k \in \mathcal{D}, (i, j) \in \mathcal{L}$ for DBM, then, there is a corresponding feasible solution satisfying the path length constraints
(\ref{o:plcon}) in OBM, provided that there is a feasible solution satisfying the path length constraints (\ref{d:plcon}) in DBM.
\end{proof}

\begin{proposition} \label{prop:plcon:otod}
There is a corresponding solution satisfying the path length constraints in DBM, for each solution satisfying the path length constraints in OBM.
\end{proposition}

\begin{proof}
Suppose ${y^*}_{ij}^{s}, {f^*}_{ij}^{s}, s \in \mathcal{S}, (i, j) \in \mathcal{L}$, ${w^*}_{ij}, (i, j) \in \mathcal{L}$, and ${l^*}_{i}^{s}, s \in \mathcal{S},
i \in \mathcal{N}$ is a feasible solution to OBM. Then since ROBM is a relaxation of OBM, ${y^*}_{ij}^{s}, {f^*}_{ij}^{s}, s \in \mathcal{S}, (i, j) \in
\mathcal{L}$ is also a feasible solution to ROBM and so satisfies all corresponding constraints.

According to the assigning process (\ref{assigning}), let ${x^*}_{ij}^{k}, k \in \mathcal{D}, (i, j) \in \mathcal{L}$ be the corresponding solution to RDBM.

$\forall s \in \mathcal{S}, \forall (i, j) \in \mathcal{L}$, there are three cases:
\begin{itemize}
\item Case 1: ${y^*}_{ij}^{s} = 0$ and $\sum_{h: (h, j) \in \mathcal{L}} {y^*}_{hj}^{s} = 0$;
\item Case 2: ${y^*}_{ij}^{s} = 0$ and $\sum_{h: (h, j) \in \mathcal{L}} {y^*}_{hj}^{s} = 1$;
\item Case 3: ${y^*}_{ij}^{s} = 1$.
\end{itemize}

For Case 1, constraints (\ref{o:plcon}) can be simplified as follows:
\begin{eqnarray*}
l^{s}_j \le l^{s}_i + w_{ij} \quad \mathrm{and} \quad l^{s}_j \ge l^{s}_i + w_{ij} - M.
\end{eqnarray*}
According to (\ref{assigning}), $\forall k \in \mathcal{D}_s, {x^*}_{ij}^{k} = 0$ and $\sum_{h: (h, j) \in \mathcal{L}} {x^*}_{hj}^{k} = 0$. Then, $\forall k \in
\mathcal{D}_s$, constraints (\ref{d:plcon}) can be simplified as follows:
\begin{eqnarray*}
l^{s_k}_j \le l^{s_k}_i + w_{ij} \quad \mathrm{and} \quad l^{s_k}_j \ge l^{s_k}_i + w_{ij} - M.
\end{eqnarray*}
Hence, the simplified constraints of (\ref{d:plcon}) are identical to those of (\ref{o:plcon}) for Case 1.

For Case 2, constraints (\ref{o:plcon}) can be simplified as follows:
\begin{eqnarray*}
l^{s}_j \le l^{s}_i + w_{ij} - \varepsilon \quad \mathrm{and} \quad l^{s}_j \ge l^{s}_i + w_{ij} - M.
\end{eqnarray*}
According to (\ref{assigning}), $\forall k \in \mathcal{D}_s, {x^*}_{ij}^{k} = 0$ and $\exists k_l \in \mathcal{D}_s, \sum_{h: (h, j) \in \mathcal{L}}
{x^*}_{hj}^{k_l} = 1$. Then, constraints (\ref{d:plcon}) can be simplified as follows:
\begin{eqnarray*}
l^{s_{k_l}}_j \le l^{s_{k_l}}_i + w_{ij} - \varepsilon \quad \mathrm{and} \quad l^{s_{k_l}}_j \ge l^{s_{k_l}}_i + w_{ij} - M,
\end{eqnarray*}
and $\forall k \in \mathcal{D}_s$ such that $\sum_{h: (h, j) \in \mathcal{L}} {x^*}_{hj}^{k} = 0$,
\begin{eqnarray*}
l^{s_k}_j \le l^{s_k}_i + w_{ij} \quad \mathrm{and} \quad l^{s_k}_j \ge l^{s_k}_i + w_{ij} - M.
\end{eqnarray*}
Hence, the simplified constraints of (\ref{d:plcon}) are identical to those of (\ref{o:plcon}) for Case 2.

For Case 3, constraints (\ref{o:plcon}) can be simplified as follows:
\begin{eqnarray*}
l^{s}_j \le l^{s}_i + w_{ij} \quad \mathrm{and} \quad l^{s}_j \ge l^{s}_i + w_{ij}.
\end{eqnarray*}
According to (\ref{assigning}), $\exists k_l \in \mathcal{D}_s, {x^*}_{ij}^{k_l} = 1$ and $\sum_{h: (h, j) \in \mathcal{L}} {x^*}_{hj}^{k_l} = 1$. In addition,
according to the sub-path optimality constraints, $\forall k \in \mathcal{D}_s$, if ${x^*}_{ij}^{k} = 0$, $\sum_{h: (h, j) \in \mathcal{L}} {x^*}_{hj}^{k} = 0$.
Then, constraints (\ref{d:plcon}) can be simplified as follows:
\begin{eqnarray*}
l^{s_{k_l}}_j \le l^{s_{k_l}}_i + w_{ij} \quad \mathrm{and} \quad l^{s_{k_l}}_j \ge l^{s_{k_l}}_i + w_{ij},
\end{eqnarray*}
and $\forall k \in \mathcal{D}_s$ such that ${x^*}_{ij}^{k} = 0$,
\begin{eqnarray*}
l^{s_k}_j \le l^{s_k}_i + w_{ij} \quad \mathrm{and} \quad l^{s_k}_j \ge l^{s_k}_i + w_{ij} - M.
\end{eqnarray*}
Hence, the simplified constraints of (\ref{d:plcon}) are identical to those of (\ref{o:plcon}) for Case 3.

Since for all the three cases, ${x^*}_{ij}^{k}, k \in \mathcal{D}, (i, j) \in \mathcal{L}$ results in the same path length constraints for DBM as those resulting
from ${y^*}_{ij}^{s}, s \in \mathcal{S}, (i, j) \in \mathcal{L}$ for OBM, then, there is a corresponding feasible solution satisfying the path length constraints
(\ref{d:plcon}) in DBM, provided that there is a feasible solution satisfying the path length constraints (\ref{o:plcon}) in OBM.
\end{proof}

\begin{corollary} \label{corol:plcon:equ}
The path length constraints in OBM are equivalent to those in DBM, concerning the feasibility of the unique shortest path routing problem.
\end{corollary}

\begin{theorem} \label{thm:feasible:equ:}
OBM and DBM and equivalent concerning the feasibility of the unique shortest path routing problem.
\end{theorem}

\begin{proof}
By Proposition~\ref{prop:d:spcon}, DBM is equivalent to DBM$^\prime$.

In addition, as discussed at the begin of Section \ref{sec:vrf}, OBM is a reduction of ROBM and DBM$^\prime$ is a reduction of RDBM. Besides the additional link
weight variables and path length variables, the difference between OBM and ROBM are the path length constraints (\ref{o:plcon}) and the difference between
DBM$^\prime$ and RDBM are the path length constraints (\ref{d:plcon}).

By Theorem~\ref{thm:feasible:equ}, ROBM and RDBM are equivalent concerning the feasibility of the relaxed problem. By Corollary~\ref{corol:plcon:equ}, the path
length constraints in OBM are equivalent to the counterparts in DBM, and so those in DBM$^\prime$. Therefore, OBM is equivalent to DBM$^\prime$, and so DBM,
concerning the feasibility of the unique shortest path routing problem.
\end{proof}

\begin{theorem} \label{thm:optimal:equ:}
OBM and DBM and equivalent concerning the optimality of the unique shortest path routing problem.
\end{theorem}

\begin{proof}
The conclusion follows directly from Theorem~\ref{thm:feasible:equ:} by constructing the corresponding optimal solutions between OBM and DBM.
\end{proof}

\begin{corollary}
OBM is a correct model of the unique shortest path routing problem.
\end{corollary}

\section{Comparisons between the Two Formulations}

Concerning the unique shortest path routing problem, it has been shown that routing performances resulting from the proposed complete formulations are much better
than those derived from the default methods, by testing on $30$ randomly generated data instances with combinations of different parameter scenarios. The
resulting average maximum utilization is $30.94\%$ of that from using the hop-count method and $45.54\%$ of that from using the inv-cap method. It hence
demonstrates the significant gain achieved by formulating the problem completely and solving it optimally.

Between the two complete formulations, compared with DBM, OBM has advantages on both constraint structure for applying constraint generation algorithms and model
size.

\subsection{Constraint Structure}

The constraint structures of DBM and OBM are shown in Figure~\ref{fig:dbm} and Figure~\ref{fig:obm}, respectively.

\begin{figure}[h!]
\begin{center}
\includegraphics [width = 0.66 \textwidth] {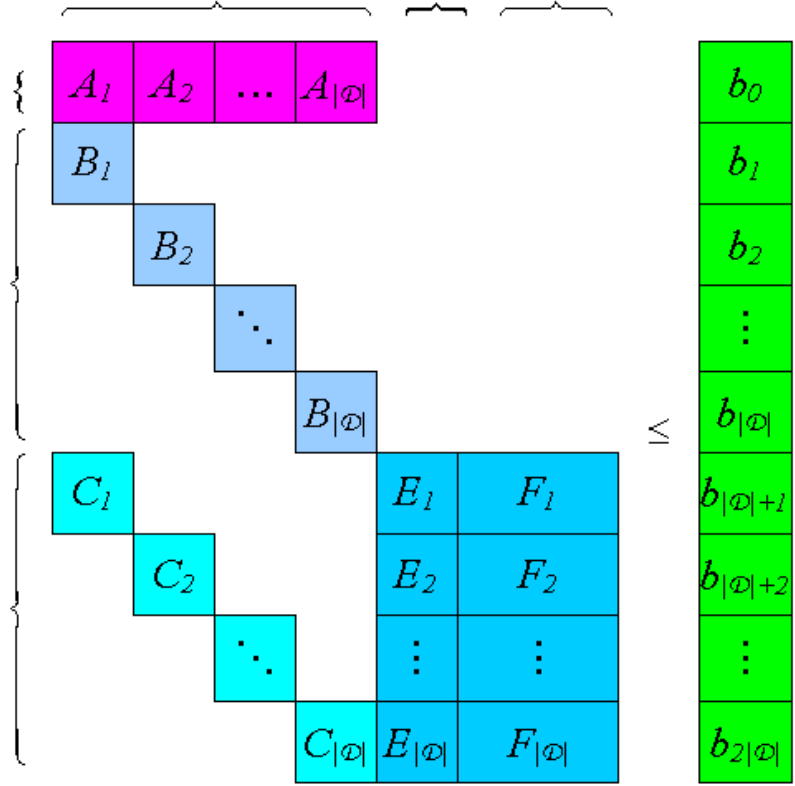}
\caption{Constraint structure of DBM} \label{fig:dbm}
\end{center}
\end{figure}
\begin{figure}[h!]
\begin{center}
\includegraphics [width = 0.66 \textwidth] {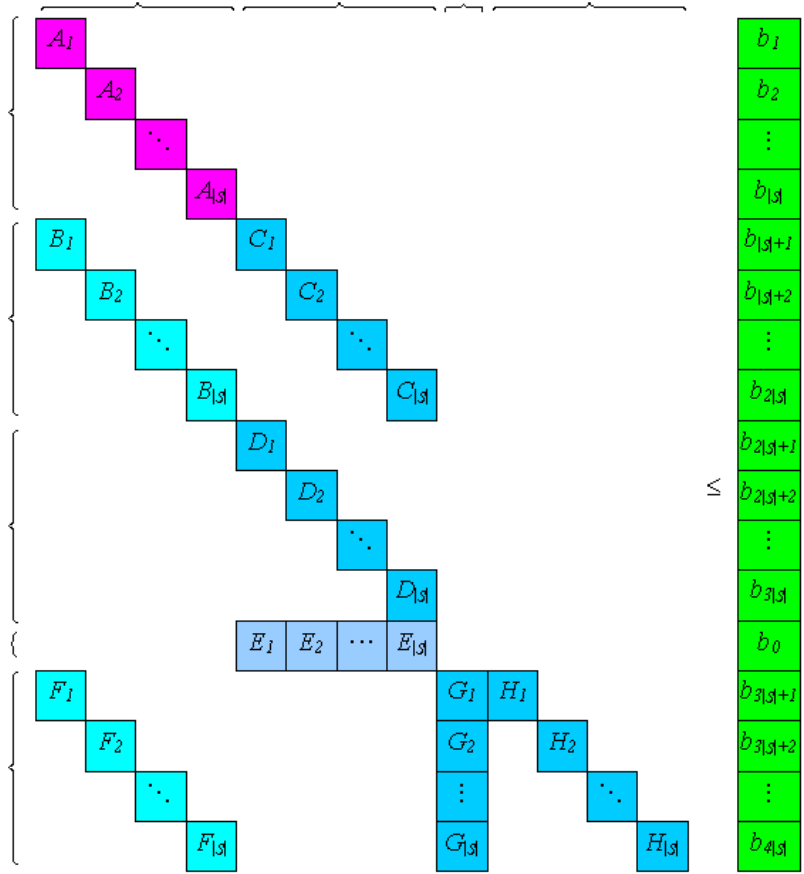}
\caption{Constraint structure of OBM} \label{fig:obm}
\end{center}
\end{figure}

In Figure~\ref{fig:dbm}, the first row represents the link capacity constraints (\ref{d:lccon}), the next four rows correspond to the flow conservation
constraints (\ref{d:fccon}), and the last four rows represent the path length constraints (\ref{d:plcon}). Accordingly, columns correspond to variables.

As can be seen, among the three sets of constraints, the flow conservation constraints and the link capacity constraints contain only the routing decision
variables, whereas the path length constraints couple the routing decision variables with the link weight variables and the path length variables. Hence,
constraint generation algorithms such as the Benders decomposition method (Benders, 1962) may be considered to be the promising solution approaches for the
problem. The problem can be decomposed into one integer programming master problem and one linear programming subproblem. The master problem deals with the flow
conservation constraints and the link capacity constraints, and so contains the routing decision variables only. Accordingly, the subproblem copes with the path
length constraints.

Similarly, in Figure~\ref{fig:obm}, the first four rows represent the path uniqueness constraints (\ref{o:pucon}), the next four rows correspond to the flow bound
constraints (\ref{o:fbcon}), the third four rows represent the flow conservation constraints (\ref{o:fccon}), the next row corresponds to the link capacity
constraints (\ref{o:lccon}), and the last four rows represent the path length constraints (\ref{o:plcon}). Columns correspond to variables accordingly.

As can be noted, although DBM has a simpler constraint structure, OBM has more flexibility to apply decomposition algorithms to solve the problem.

As shown in Figure \ref{fig:obm}, with OBM, the problem can be globally decomposed into one master problem and two subproblems, instead of one master problem and
one subproblem as with DBM. The master problem contains only the routing decision variables and the path uniqueness constraints accordingly. The first subproblem
deals with the auxiliary flow variables and the second subproblem copes with the link weight variables and the path length variables. In addition, the master
problem can be further decomposed, with one independent subproblem corresponding to each origin node.

\subsection{Model Size}

Compared with DBM, OBM defines explicitly the auxiliary flow variables and the flow bound constraints accordingly. However, in general, $|\mathcal{S}| <<
|\mathcal{D}|$, and the size of OBM is much smaller than that of DBM. The model sizes of the two formulations are as shown in Table~\ref{table:size}, where
\textit{\#Variables} represents the number of variables and \textit{\#Constraints} denotes the number of constraints.

More concretely, the model sizes of both the original problems and the master problems of the two formulations on a randomly generated data instance with
$|\mathcal{N}| = 50$, $|\mathcal{L}| = 642$, $|\mathcal{D}| = 1000$, and $|\mathcal{S}| = 50$ are shown in Table~\ref{table:instance}.

As can be seen from Table~\ref{table:instance}, with OBM, the number of variables of the original problem decreases from over $600,000$ to $64,200$ and the number
of constraints drops from over $1,000,000$ to less than $38,000$. In addition, with OBM, both the number of variables and the number of constraints of the master
problem decline $20$ times.
\begin{table}[h!]
\begin{center}
\caption{Model sizes of DBM and OBM} \label{table:size}
\begin{tabular}{c r r r r}
\hline
Model & & \multicolumn{1}{c}{\#Variables} & & \multicolumn{1}{c}{\#Constraints} \\
\hline
DBM & & $|\mathcal{D}||\mathcal{L}| + |\mathcal{S}||\mathcal{N}| + |\mathcal{L}|$
& & $|\mathcal{D}||\mathcal{N}| + 2|\mathcal{D}||\mathcal{L}| + |\mathcal{L}|$ \\
OBM & & $2|\mathcal{S}||\mathcal{L}| + |\mathcal{S}||\mathcal{N}| + |\mathcal{L}|$
& & $2|\mathcal{S}||\mathcal{N}| + 3|\mathcal{S}||\mathcal{L}| + |\mathcal{L}|$ \\
\hline
\end{tabular}
\end{center}
\end{table}
\begin{table}[h!]
\begin{center}
\caption{Model sizes of DBM and OBM of a large data instance}
\label{table:instance}
\begin{tabular}{c r r r r r r}
\hline
& & \multicolumn{2}{c}{Original Problem} & & \multicolumn{2}{c}{Master Problem} \\
\cline{3-4} \cline{6-7}
& & \multicolumn{1}{c}{\#Variables} & \multicolumn{1}{c}{\#Constraints}
& & \multicolumn{1}{c}{\#Variables} & \multicolumn{1}{c}{\#Constraints} \\
\hline
DBM & & $645,142 \quad$ & $1,334,642 \quad$ & & $642,000 \quad$ & $50,642 \quad$ \\
OBM & & $64,200 \quad$ & $37,742 \quad$ & & $32,100 \quad$ & $2,500 \quad$ \\
\hline
\end{tabular}
\end{center}
\end{table}

As a conclusion, compared with DBM, OBM has a smaller model size and a more flexible constraint structure for decomposition algorithms such as the Benders
decomposition method to solve the problem.

\section{Conclusions}

With the aim of an exact solution approach to the unique shortest path routing problem on average data instances arising from real-world applications, two
complete and explicit mathematical formulations with a polynomial number of constraints for the problem are developed. A demand-based formulation is first
introduced, based on the study of the relationships between the length of a shortest path and the weights of links that the path traverses. The problem is further
formulated as an origin-based model by analyzing solution properties of the problem. The two formulations are then mathematically proved to be correct and to be
equivalent concerning both the feasibility and the optimality of the problem. Based on the study of the constraint structures and model sizes of the two
formulations, the origin-based formulation is identified to be the better one for decomposition algorithms such as the Benders decomposition method to solve the
problem.

The two formulations may be generalized to other network flow and network routing problems. Prospective future work may lie in investigating possible improvements
concerning both problem formulation and solution algorithm to improve the efficiency of the solution approach proposed. In particular, investigations may focus on
possible improvements related to three factors: the closeness between the initial solution and the final solution to the master problem, the strength of cuts
generated at each iteration, and the efficiency of an algorithm to solve the integer programming master problem. For example, redundant constraints may be
generated to tighten the feasible region of the initial master problem, strategies such as active set method may be applied to strengthen the cuts generated from
the subproblems, and schemes such as the Lagrangian relaxation method may be embedded into the solution algorithm to improve the efficiency of solving the master
problem at each iteration.

\section*{References}


\begin{hangref}

\item Ahuja, R. K., T. L. Magnanti, J. B. Orlin. 1993. \textit{Network Flows: Theory, Algorithms, and Applications}. Prentice Hall.
%
\item Ameur, W. B., E. Gourdin. 2003. Internet Routing and Related Topology Issues. \textit{SIAM J. Discrete Math.} \textbf{17} 18--49.
%
\item Benders, J. 1962. Partitioning Procedures for Solving Mixed-Variables Programming Problems. \textit{Numer. Math.} \textbf{4} 238--252.
%
\item Bertsekas, D., R. Gallager. 1992. \textit{Data Networks}. Prentice Hall.
%
\item Bley, A., T. Koch. 2002. Integer Programming Approaches to Access and Backbone IP-network Planning. ZIB-Report 02-41.
%
\item Ericsson, M., M. G. C. Resende, P. M. Pardalos. 2002. A Genetic Algorithm for the Weight Setting Problem in OSPF Routing. \textit{J.
Comb. Optim.} \textbf{6} 299--333.
%
\item Feldmann, A., A. Greenberg, C. Lund, N. Reingold, J. Rexford, F. True. 2001. Deriving Traffic Demands for Operational IP Networks:
Methodology and Experience. \textit{IEEE/ACM Trans. Netw.} \textbf{9} 265--280.
%
\item Fortz, B., M. Thorup. 2000. Internet Traffic Engineering by Optimizing OSPF Weights. \textit{Proc. of 19th IEEE Conference on Computer
Communications} 519--528.
%
\item Holmberg, K., D. Yuan. 2004. Optimization of Internet Protocol Network Design and Routing. \textit{Networks} \textbf{43} 39--53.
%
\item Lin, F. Y. S., J. L. Wang. 1993. Minimax Open Shortest Path First Routing Algorithms in Networks Supporting the SMDS Services.
\textit{Proc. of IEEE International Conference on Communications}
666--670.
%
\item Moy, J. 1998. \textit{OSPF Anatomy of an Internet Routing Protocol}. Addison-Wesley.
%
\item Ramakrishnan, K. G., M. A. Rodrigues. 2001. Optimal Routing in Shortest-Path Data Network. \textit{Bell Labs Technical Journal}
\textbf{6} 117--138.
%
\item Zhang, C., R. Rodo\v sek. 2005. Modelling and Constraint Hardness Characterisation of the Unique-Path OSPF Weight Setting Problem.
V.S. Sunderam et al., Eds. \textit{ICCS 2005, LNCS 3514}. Springer-Verlag, Berlin Heidelberg. 804--811.

\end{hangref}


\end{document}